\documentclass[11pt]{amsart}
\usepackage{amscd,amssymb,amsmath}
\usepackage[all]{xy}

\newtheorem{theorem}{Theorem}[section]
\newtheorem{lemma}[theorem]{Lemma}
\newtheorem{Lemma}[theorem]{Lemma}

\theoremstyle{definition}

\newtheorem{example}[theorem]{Example}
\newtheorem{proposition}[theorem]{Proposition}
\newtheorem{corollary}[theorem]{Corollary}

\theoremstyle{remark}
\newtheorem{remark}{Remark}[section]

\numberwithin{equation}{section}
\numberwithin{table}{section}
\newcommand{\zhclass}[1]{ \langle #1 \rangle}
\newcommand{\im}{\mathrm{Im}}

\newcommand{\tmod}[1]{\; (#1)}

\newcommand{\bbZ}{\mathbb{Z}}
\newcommand{\dg}{\overline{\deg}}

\begin{document}

\title[Mapping degrees between spherical $3$-manifolds]{Mapping degrees between spherical $3$-manifolds}
\author{Daciberg Gon\c calves}
\address{Dept. de Matem\'atica - IME - USP, Caixa Postal 66.281 - CEP 05314-
 970,
 S\~ao Paulo - SP, Brasil; FAX: 55-11-30916183}
\email{dlgoncal@ime.usp.br}

\author{Peter Wong}
\address{Department of Mathematics, Bates College, Lewiston,
ME 04240, U.S.A.; FAX: 1-207-7868331}
\email{pwong@bates.edu}

\author{Xuezhi Zhao}
\address{Department of Mathematics, Capital Normal University, Beijing 100048, China; FAX: 86-10-68900950}
\email{zhaoxve@mail.cnu.edu.cn}

\thanks{This work was initiated during the first and second authors' visit to Capital Normal University February 16 - 23, 2013.  The first author was supported in part by  Projeto Tem\'atico  Topologia Alg\'ebrica Geom\'etrica e Diferencial  2012/24454-8. The third author was supported in part by the NSF of China (11431009).}

\begin{abstract}
Let $D(M,N)$ be the set of integers that can be realized as the degree of a
map between two closed connected orientable manifolds $M$ and $N$ of the
same dimension. For closed $3$-manifolds with $S^3$-geometry $M$ and $N$,
every such degree $\deg f\equiv \overline {\deg}\psi$ $(|\pi_1(N)|)$ where
$0\le \overline {\deg}\psi <|\pi_1(N)|$ and $\overline {\deg}\psi$ only
depends on the induced homomorphism $\psi=f_{\pi}$ on the fundamental
group. In this paper, we calculate   explicitly the set $\{\overline
{\deg}\psi\}$ when $\psi$ is  surjective and  then we show how to 
determine  $\overline{\deg}(\psi)$   for arbitrary  
 homomorphisms. This leads to the determination of the set $D(M,N)$.
\end{abstract}
\date{\today}

\maketitle

\newcommand{\af}{\alpha}
\newcommand{\et}{\eta}
\newcommand{\ga}{\gamma}
\newcommand{\ta}{\tau}
\newcommand{\ph}{\varphi}
\newcommand{\bt}{\beta}
\newcommand{\lb}{\lambda}
\newcommand{\wh}{\widehat}
\newcommand{\sg}{\sigma}
\newcommand{\om}{\omega}
\newcommand{\cH}{\mathcal H}
\newcommand{\cF}{\mathcal F}
\newcommand{\N}{\mathcal N}
\newcommand{\R}{\mathcal R}
\newcommand{\Ga}{\Gamma}
\newcommand{\cc}{\mathcal C}
\newcommand{\bea} {\begin{eqnarray*}}
\newcommand{\beq} {\begin{equation}}
\newcommand{\bey} {\begin{eqnarray}}
\newcommand{\eea} {\end{eqnarray*}}
\newcommand{\eeq} {\end{equation}}
\newcommand{\eey} {\end{eqnarray}}
\newcommand{\ovl}{\overline}
\newcommand{\vv}{\vspace{4mm}}
\newcommand{\lra}{\longrightarrow}

\bibliographystyle{amsplain}

\section{Introduction}

The study of mapping degree has a very long history in topology. In particular, the existence of degree one maps between manifolds is an important aspect in the classification of manifolds.
Given two closed connected orientable manifolds $M$ and $N$ of the same dimension, let $D(M,N)=\{\deg f|f:M\to N\}$ denote the possible values of mapping degrees of maps from $M$ to $N$. While the determination of $D(M,N)$ has been of great interest to topologists, in general, very little has been done concerning the computation of $D(M,N)$. For example, the finiteness of the set $D(M,N)$ has been studied by various authors (see e.g. \cite{DSW,DP, DW1, DW2, LX,A}) and the study of $D(M):=D(M,M)$ for self-maps have also been investigated. We refer the reader to the survey paper \cite{Wang} for the study of degree up to 2000, and to the papers \cite{Du,SWW,SWWZ} and \cite{HKWZ} for mapping degrees on $3$-manifolds that are most relevant to our work.
Prompted by our earlier work \cite{GWZ3} on the fixed point theory of spherical $3$-manifolds, in this paper we determine $D(M,N)$ when $M$ and $N$ are closed $3$-manifolds with $S^3$ geometry.

\section{Fundamental groups of spherical $3$-manifolds}

A spherical $3$-manifold is an orbit space $S^3/G$, where $G$ is a finite group which acts freely on $S^3$. Hence $G$ is the fundamental group of $S^3/G$. Each spherical $3$-manifold is oriented by the canonical orientation of $S^3$.

\begin{proposition}\label{presentation} (cf. \cite{Scott})
The fundamental group of any spherical $3$-manifold is in the form of $G$ or $\bbZ_m\times G$,  where $G$ belongs to one of the groups listed below and $(m, |G|)=1$.
\begin{table}[h]\label{basic_pi1}
\begin{center}
\begin{tabular}{|l|l|l|l l|}
\hline
group  & condition & presentation & normal form & \\ \hline
$\mathbb{Z}_n$ &
 & $\zhclass{c \mid c^n=1}$
 & $c^s$, & $0\le s< n$
 \\ \hline
$D^*_{4n}$ & $2|n$
  & $\zhclass{ b, a \mid a^2 = b^n= (ab)^2, a^4 =1 }$
  &  $b^s$, $b^s a$, & $0 \le s< 2n$
\\ \hline
$O^*_{48}$ &
  & $\zhclass{ b, a \mid a^2 = b^3= (ab)^4, a^4 =1}$
  & see Lemma~\ref{conjTo48} &
\\ \hline
$I^*_{120}$ &
  & $\zhclass{ b, a \mid a^2 = b^3= (ab)^5, a^4 =1}$
  & see Lemma~\ref{conjTi120} &
\\ \hline
$T'_{8\cdot 3^q}$ & $q\ge 1$
 & \!\!\!$\begin{array}{ll}
 \langle b, a, w  \mid  a^2=b^2=(ab)^2,\\
  \ \ \ \ \ \ \ \ \ \ \ \ a^4=w^{3^q}=1,\\
  \ \ \ \ \ \ \ \ \ \ \ \ wa=bw, wb=abw \rangle
 \end{array}$
  & $b^s w^t$, $b^s a w^t$, &
  $
  0\le s<4,  \ 0\le t <3^q
  $
  \\   \hline
$D'_{n\cdot 2^q}$\ & $2\not|n, n>1, q>1$
 & $\zhclass{ u, w\mid u^{n}=w^{2^q} =1, uwu=w}$
 & $u^sw^t$ &
  $
  0\le s <n,  \ 0\le t <2^q
  $
 \\ \hline
\end{tabular}
\end{center}
\caption{types of fundamental groups of spherical $3$-manifolds}
\end{table}
\end{proposition}

The binary tetrahedral group is denoted by $T^*_{24} =\{ a, b \mid a^2 = b^3= (ab)^3, a^4 =1\}$, which is actually a special case of the generalized tetrahedral group $T'_{8\cdot 3^q}$ with $q=1$, i.e. $T^*_{24} \cong T'_{8\cdot 3^1}$. One such isomorphism from $T'_{8\cdot 3^1}$ to $T^*_{24}$ is given by  $b\mapsto bab^{-1}, a\mapsto a, w\mapsto ba^{-1}$.

If $n$ is odd, then the generalized dihedral group $D^*_{4n}$ is the same as the dicyclic group $D'_{n\cdot 2^2}$. If $n=1$, then $D'_{n\cdot 2^q}$ is the cyclic group $\bbZ_{2^q}$. If $q=1$, then $D'_{n\cdot 2^q}$ is not the fundamental group of a $3$-dimensional spherical manifold except for the case $n=1$. Thus, all groups in Table \ref{basic_pi1} are distinct.

Here are some more characteristic subgroups and quotients of these groups.

\begin{table}[h]\label{char_subgp}
\begin{center}
\begin{tabular}{|l|l|l|l|l|}
\hline
group  &  commutators & abelianization  & center \\ \hline
$\mathbb{Z}_n$
 & $1$
 & $\langle c \rangle \cong \mathbb{Z}_n$
 & $\langle c \rangle \cong \mathbb{Z}_n$
  \\ \hline
$D^*_{4n}$\ ($2|n$)
  & $\langle b^2 \rangle \cong \mathbb{Z}_n$
  & $\langle  \bar b,  \bar a \rangle \cong \mathbb{Z}_2\oplus\mathbb{Z}_2$
  & $\langle b^n \rangle \cong \mathbb{Z}_2$
\\ \hline
$O^*_{48}$
  & $T^*_{24}$
  & $\mathbb{Z}_2$
  & $\langle a^2 \rangle \cong \mathbb{Z}_2$
\\ \hline
$I^*_{120}$
  & $I^*_{120}$
  & $1$
  & $\langle a^2 \rangle \cong \mathbb{Z}_2$
\\ \hline
$T'_{8\cdot 3^q}$
 & $\langle a,b \rangle\cong D^*_{4\cdot 2}$
 & $\langle w \rangle \cong \mathbb{Z}_{3^q}$
 & $\langle b^2, w^3 \rangle \cong \mathbb{Z}_{2\cdot 3^{q-1}}$
 \\ \hline
$D'_{n\cdot 2^q}$\ ($2\not|n$)
 & $\langle u \rangle\cong \mathbb{Z}_{n}$
 & $\langle w \rangle \cong \mathbb{Z}_{2^q}$
 & $\langle  w^2 \rangle \cong \mathbb{Z}_{2^{q-1}}$
  \\ \hline
\end{tabular}
\end{center}
\caption{characteristic subgroups}
\end{table}

Next, we describe, for each group in Table \ref{basic_pi1} except for the cyclic groups, the conjugacy classes of elements, subgroups and quotients.

%\newpage

\begin{Lemma}\label{conjD4n}
The set of conjugacy classes of $D^*_{4n}$ is given as follows.

\begin{table}[h]\label{cc-D*}
\begin{center}
\begin{tabular}{|l|c|l|}
\hline
  representative & order & all elements \\ \hline
  $[1]$  &  $1$ &  $\{1\}$  \\ \hline
  $[b^s]$\   $0< s <n$ &  $\frac{2n}{(n,s)}$ & $\{b^s, b^{2n-s}\}$  \\ \hline
  $[b^n]$ & $2$  &  $\{b^n\}$   \\ \hline
  $[a]$ &  $4$ &   $\{b^{2s}a\mid 0\le s<n \}$  \\ \hline
  $[ba]$ &  $4$ &   $\{b^{2s+1}a\mid 0\le s<n \}$  \\ \hline
\end{tabular}
\end{center}
\caption{conjugacy classes of $D^*_{4n}$}
\end{table}
\end{Lemma}

\begin{proof}
Note that $b(b^ka)b^{-1} = b^{k+2}a$, $ab^la^{-1} = b^{-l}$, $a(b^ka)a^{-1} = b^{-k}a$.
\end{proof}

\begin{Lemma}\label{subgroupD4n}
The non-trival subgroups and quotients of $D^*_{4n}$ are listed as follows.

\begin{table}[h]
\begin{center}
\begin{tabular}{|l|l|l|}
\hline
  subgroup type & conjugacy classes  & quotient \\ \hline
 $\bbZ_{\frac{2n}{k}}$  ($k|n$) &  $\{\zhclass{b^k}\}$
    &  $D_{2k}$ \\ \hline
 $\bbZ_{\frac{2n}{k}}$  ($k|2n$, $k\not|n$) &  $\{\zhclass{b^k}\}$
    &  $D^*_{4\cdot \frac{k}{2}}$  \\ \hline
  $\bbZ_{4}$   &  $\{\zhclass{b^{2j}a}\mid 0\le j<n \}$ &  $\bbZ_2$ if $n=2$ \\ \hline
  $\bbZ_{4}$   &  $\{\zhclass{b^{2j+1}a}\mid 0\le j<n \}$ &  $\bbZ_2$ if $n=2$ \\ \hline
 $D^*_{4\cdot\frac{n}{k}}$  ($k|n$, $2|\frac{n}{k}$) &  $\{\zhclass{b^k, b^{2j}a}\mid 0\le j<n \}$
    &   $\bbZ_2$ if $k=2$   \\ \hline
$D'_{\frac{n}{k}\cdot 2^2}$  ($k|n$, $2\not|\frac{n}{k}$)
  &  $\{\zhclass{b^k, b^{2j}a}\mid 0\le j<n \}$
  &   $\bbZ_2$ if $k=2$  \\ \hline
 $D^*_{4\cdot\frac{n}{k}}$  ($k|n$, $2|\frac{n}{k}$) &  $\{\zhclass{b^k, b^{2j+1}a}\mid 0\le j<n \}$
    &   $\bbZ_2$ if $k=2$   \\ \hline
$D'_{\frac{n}{k}\cdot 2^2}$  ($k|n$, $2\not|\frac{n}{k}$)
  &  $\{\zhclass{b^k, b^{2j+1}a}\mid 0\le j<n \}$
  &   $\bbZ_2$ if $k=2$  \\ \hline
\end{tabular}
\end{center}
\caption{subgroups and quotients of $D^*_{4n}$}
\end{table}
\end{Lemma}

\begin{proof}
The presentations of cyclic subgroups follow from Lemma~\ref{conjD4n}. Since $a^2=b^n$, any non-cyclic subgroup $H$ of $D^*_{4n}$ must have the form of either $\zhclass{b^k, b^la}$ or $\zhclass{b^ka, b^la}$. Note that $b^la$ is conjugate to either $a$ or $ba$. There are four possibilities: $\zhclass{b^k, a}$, $\zhclass{b^ka, a}$, $\zhclass{b^k, ba}$ and $\zhclass{b^ka, ba}$.

Clearly, $\zhclass{b^ka, a}=\zhclass{b^k, a}$. Since $b$ has order $2n$, we may assume that $k|2n$. If $\frac{2n}{k}$ is even, i.e. $\frac{n}{k}$ is an integer, then we have the presentation $\zhclass{b^k, a} = \{b^k, a\mid a^2 = (b^k)^{\frac{n}{k}}= (ab^k)^2, a^4 =1\}$. Thus $\zhclass{b^k, a}$ is isomorphic to $D^*_{4\cdot\frac{n}{k}}$ if $\frac{n}{k}$ is even; to $D'_{\frac{n}{k}\cdot 2^2}$ if $\frac{n}{k}$ is odd. In the case that $\frac{2n}{k}$ is odd, we have that $(n,k)=\frac{k}{2}$. Since $a^2=b^n$, the subgroup $\zhclass{b^k, a} = \zhclass{b^k, b^n, a} = \zhclass{b^{\frac{k}{2}},  a}$ is isomorphic to $D'_{\frac{2n}{k}\cdot 2^2}$ by the argument above.

Suppose that $\zhclass{b^k, a}$ is normal. Then $bab^{-1}=b^2a$ must lie in $\zhclass{b^k, a}$. It follows that $b^2\in \zhclass{b^k, a}$. Hence $k$ must be $1$ or $2$. Note that $\zhclass{b, a}$ is the whole group $D^*_{4n}$. Thus, $k=2$ is the unique non-trivial case that $\zhclass{b^k, a}$ is normal.

The subgroup $\zhclass{b^k, ba}$ is isomorphic to $\zhclass{b^k, a}$, by using the automorphism given by $a\mapsto b^{-1}a,\ b\mapsto b$. This situation is the same as that of $\zhclass{b^k, a}$.

Note that $\zhclass{b^ka, ba} = \zhclass{b^{k-1}, ba}$. This is also the case we have already considered.
\end{proof}

Next, we consider the binary octahedral group $O^*_{48}$ and the binary icosahedral group $I^*_{120}$. The subgroups and quotients of these groups are well-known and the results can be found, for example, in \cite[Appendix]{GG}.

\begin{Lemma}\label{subgroupO48}
The subgroups of $O^*_{48}$ are $\mathbb \bbZ_2, \mathbb \bbZ_3, \mathbb \bbZ_4, \mathbb \bbZ_6, \mathbb \bbZ_8$, $Q_8$, $D'_{3\cdot 2^2}$, $Q_{16}$ and $T^*_{24}$.

The normal subgroups are $ \mathbb \bbZ_2,  Q_8 \  and \  T^*_{24}$.

The only quotient  which is  the fundamental group of a spherical $3$-manifold is  $\mathbb \bbZ_2$, which can be embedded in $O^*_{48}$.
\end{Lemma}

\begin{Lemma}\label{conjTo48}
The set of conjugacy classes of $O^*_{48}$ is given as follows.

\begin{table}[h]
\begin{center}
\begin{tabular}{|l|c|l|}
\hline
  class & order & all elements \\ \hline

 $[a]$ & $4$ & $\{ a, a^{-1},bab^{-1}, ba^{-1}b^{-1}, b^{-1}ab, b^{-1}a^{-1}b, abab^{-1}a$,\\
 & &  \ \ $abab^{-1}a^{-1},ab^{-1}aba,ab^{-1}aba^{-1},bab^{-1}aba,bab^{-1}aba^{-1}\}$ \\  \hline
  $[baba]$ & $4$ & $\{baba, baba^{-1},b^{-1}ab^{-1}a,b^{-1}ab^{-1}a^{-1},
    bab^{-1}ab,bab^{-1}a^{-1}b\}$ \\  \hline
 $[b^2]$ & $3$ &  $\{b^{2},b^{-2},aba,ab^{-1}a,bab^{-1}a,b^{-1}aba, abab^{-1},ab^{-1}ab\}$ \\  \hline
  $[b]$ & $6$ & $\{b,b^{-1},aba^{-1},ab^{-1}a^{-1},bab^{-1}a^{-1},b^{-1}aba^{-1},aba^{-1}b^{-1},
  ab^{-1}a^{-1}b\}$  \\ \hline
  $[ba]$ & $8$ & $\{ba,b^{-1}a^{-1},ab,a^{-1}b^{-1},bab,b^{-1}a^{-1}b^{-1}\}$ \\ \hline
  $[ba^{-1}]$ & $8$ & $\{ba^{-1},b^{-1}a,ab^{-1},a^{-1}b,ba^{-1}b,b^{-1}ab^{-1}\}$ \\ \hline
  $[a^2]$ & $2$ & $\{a^{2}\}$ \\ \hline
   $[1]$ & $1$ & $\{1\}$ \\ \hline
\end{tabular}
\end{center}
\caption{conjugacy classes of $O^*_{48}$}
\end{table}
\end{Lemma}

\begin{Lemma}\label{subgroupI120}
The subgroups of $I^*_{120}$ are  $   \mathbb \bbZ_2, \mathbb \bbZ_3, \mathbb \bbZ_4, \mathbb \bbZ_5, \mathbb \bbZ_6$, $Q_8$, $\mathbb \bbZ_{10}$, $D'_{12}$, $D'_{20}$ and   $T^*_{24}$.

 The only normal subgroup is $ \mathbb \bbZ_2$.

The only quotient  is $A_5$ which is not the fundamental group of a spherical $3$-manifold.
\end{Lemma}

\newpage

\begin{Lemma}\label{conjTi120}
The set of conjugacy classes of $I^*_{120}$ is given as follows.
\begin{table}[h]
\begin{center}
\begin{tabular}{|l|c|l|}
\hline
  class & order & all elements \\ \hline
$[a]$ & $4$  &  $
\begin{array}{l}
\{a,a^{-1},bab^{-1},ba^{-1}b^{-1},b^{-1}ab,b^{-1}a^{-1}b,abab^{-1}a,abab^{-1}a^{-1},
ab^{-1}aba, \\
 \ ab^{-1}aba^{-1},babab^{-1}ab^{-1},babab^{-1}a^{-1}b^{-1},bab^{-1}abab^{-1},
bab^{-1}aba^{-1}b^{-1},\\
 \ b^{-1}abab^{-1}ab,b^{-1}abab^{-1}a^{-1}b,b^{-1}ab^{-1}abab,b^{-1}ab^{-1}aba^{-1}b,
babab^{-1}aba,\\
 \ babab^{-1}aba^{-1},b^{-1}ab^{-1}abab^{-1}a,b^{-1}ab^{-1}abab^{-1}a^{-1},
bab^{-1}abab^{-1}ab,\\
 \ bab^{-1}abab^{-1}a^{-1}b, abab^{-1}abab^{-1}a,abab^{-1}abab^{-1}a^{-1},
ab^{-1}abab^{-1}aba,\\
 \ ab^{-1}abab^{-1}aba^{-1},bab^{-1}abab^{-1}aba,bab^{-1}abab^{-1}aba^{-1}\}
\end{array}
$  \\ \hline
$[baba]$ & $5$ &
$
\begin{array}{l}
\{baba,b^{-1}ab^{-1}a,abab,ab^{-1}ab^{-1},bab^{-1}a^{-1}b,b^{-1}aba^{-1}b^{-1},abab^{-1}aba,\\
\ ab^{-1}abab^{-1}a,bab^{-1}abab^{-1}a,b^{-1}abab^{-1}aba,abab^{-1}abab^{-1},ab^{-1}abab^{-1}ab\}
\end{array}
$
 \\ \hline
$[baba^{-1}]$ & $10$ &
$\begin{array}{l}
\{baba^{-1},b^{-1}ab^{-1}a^{-1},aba^{-1}b,ab^{-1}a^{-1}b^{-1},bab^{-1}ab,b^{-1}abab^{-1},
abab^{-1}aba^{-1},\\
 \ ab^{-1}abab^{-1}a^{-1},bab^{-1}abab^{-1}a^{-1},b^{-1}abab^{-1}aba^{-1},
abab^{-1}aba^{-1}b^{-1},\\
 \ ab^{-1}abab^{-1}a^{-1}b\}
\end{array}
$  \\ \hline
$[b^{2}]$ &  $3$ &
$\begin{array}{l}
\{b^{2},b^{-2},aba,ab^{-1}a,babab^{-1},bab^{-1}ab^{-1},b^{-1}abab,b^{-1}ab^{-1}ab,b
abab^{-1}a,\\
 \ bab^{-1}aba,b^{-1}abab^{-1}a^{-1},b^{-1}ab^{-1}aba^{-1},abab^{-1}ab,
abab^{-1}a^{-1}b^{-1}, \\
 \ ab^{-1}abab,ab^{-1}aba^{-1}b^{-1},babab^{-1}ab, bab^{-1}abab,
b^{-1}abab^{-1}a^{-1}b^{-1},\\
 \ b^{-1}ab^{-1}aba^{-1}b^{-1}\}
\end{array}
$ \\ \hline
$[b]$ & $6$  &
$\begin{array}{l}
\{b,b^{-1},aba^{-1},ab^{-1}a^{-1},
baba^{-1}b^{-1},bab^{-1}a^{-1}b^{-1},b^{-1}aba^{-1}b,b^{-1}ab^{-1}a^{-1}b,\\
 \ babab^{-1}a^{-1},
bab^{-1}aba^{-1},b^{-1}abab^{-1}a,b^{-1}ab^{-1}aba,abab^{-1}ab^{-1},abab^{-1}a^{-1}b,\\
 \ ab^{-1}abab^{-1},ab^{-1}aba^{-1}b,babab^{-1}a^{-1}b,bab^{-1}aba^{-1}b,b^{-1}abab^{-1}ab^{-1},\\
 \ b^{-1}ab^{-1}abab^{-1}\},
\end{array}
$ \\ \hline
$[ba]$ & $10$ &
$\begin{array}{l}
\{ba,b^{-1}a^{-1},ab,a^{-1}b^{-1},bab,b^{-1}a^{-1}b^{-1},
bab^{-1}a^{-1},b^{-1}aba^{-1},aba^{-1}b^{-1},\\
 \ ab^{-1}a^{-1}b,babab,b^{-1}ab^{-1}ab^{-1}\}
\end{array}
$ \\ \hline
$[ba^{-1}]$ & $5$ & $\begin{array}{l}
\{ba^{-1},b^{-1}a,ab^{-1},a^{-1}b,ba^{-1}b,b^{-1}ab^{-1},bab^{-1}a,b^{-1}aba,abab^{-1},ab^{-1}ab,\\
\ baba^{-1}b,b^{-1}ab^{-1}a^{-1}b^{-1}\}
\end{array}
$ \\ \hline
$[a^{2}]$ & $2$ & $\{a^{2}\}$ \\ \hline
$[1]$ & $1$ & $\{1\}$ \\ \hline

\end{tabular}
\end{center}
\caption{conjugacy classes of $I^*_{120}$}
\end{table}
\end{Lemma}

\newpage

\begin{Lemma}\label{conjT8}
The set of conjugacy classes of $T'_{8\cdot 3^q}$ is given as follows.
\begin{table}[h]
\begin{center}
\begin{tabular}{|c|l|c|l|}
\hline
  & representative & order & all elements \\ \hline
  1 & $[w^{3t}]$  & $\frac{3^q}{(3^q, 3t)}$ &  $\{w^{3t}\}$ \\ \hline
 2 & $[bw^{3t}]$  &  $\frac{4\cdot 3^q}{(3^q, 3t)}$
   & $\{bw^{3t}, b^3w^{3t}, aw^{3t}, baw^{3t}, b^2aw^{3t}, b^3aw^{3t}\}$ \\ \hline
 3 &  $[b^2w^{3t}]$  &  $\frac{2\cdot 3^q}{(3^q, 3t)}$  & $\{b^2w^{3t}\}$ \\ \hline
 4 & $[w^{3t+1}]$  & $3^q$ & $\{w^{3t+1}, b^3w^{3t+1}, baw^{3t+1}, b^2aw^{3t+1}\}$ \\ \hline
 5 & $[bw^{3t+1}]$  & $2\cdot 3^q$ & $\{bw^{3t+1}, b^2w^{3t+1}, aw^{3t+1}, b^3aw^{3t+1}\}$ \\ \hline
 6 & $[w^{3t+2}]$  &  $3^q$  & $\{w^{3t+2}, bw^{3t+2}, aw^{3t+2}, b^3aw^{3t+2}\}$ \\ \hline
 7 & $[b^2w^{3t+2}]$  &  $2\cdot 3^q$  & $\{b^2w^{3t+2}, b^3w^{3t+2}, baw^{3t+2}, b^2aw^{3t+2}\}$ \\ \hline
\end{tabular}
\end{center}
\caption{conjugacy classes of $T'_{8\cdot 3^q}$}
\end{table}
Here, $t=0,1,\ldots, 3^{q-1}-1$.
\end{Lemma}

\begin{proof}
Since $bwb^{-1}= b (ab)^{-1} w = a^{-1} w$ and  $a w a^{-1} = a  b^{-1} w =baw$, two conjugate elements must lie in the same coset of $[T'_{8\cdot 3^q},T'_{8\cdot 3^q}]$. Note that $w^3$ commutes with all elements of $T'_{8\cdot 3^q}$. It is sufficient to consider the conjugate relations in cosets $[T'_{8\cdot 3^q},T'_{8\cdot 3^q}]$, $[T'_{8\cdot 3^q},T'_{8\cdot 3^q}]w$ and $[T'_{8\cdot 3^q},T'_{8\cdot 3^q}]w^2$.

(1) The coset $[T'_{8\cdot 3^q},T'_{8\cdot 3^q}]$. Since  $[T'_{8\cdot 3^q},T'_{8\cdot 3^q}]$ is isomorphic $D^*_{4\cdot 2}$, from Lemma~\ref{conjD4n}, we have that $b\sim b^3$, $a\sim b^2a$ and $ba\sim b^3a$ by conjugation induced by $b$ and $a$. Note that $waw^{-1}= b$ and $wbw^{-1}= ab = b^3a$.  Thus, the coset $[T'_{8\cdot 3^q},T'_{8\cdot 3^q}]$ contains three conjugacy classes: $[1]=\{1\}$, $[b]=\{ b, b^3, a, ba, b^2a, b^3a\}$ and $[b^2]=\{b^2\}$.

(2) The coset $[T'_{8\cdot 3^q},T'_{8\cdot 3^q}] w$.  Note that
$$b(b^s a w)b^{-1} = b^{s+1} a (ab)^{-1} w = b^{s+2} w,\
b(b^s w)b^{-1} = b^{s+1} (ab)^{-1} w = b^{s+2}a w,$$
$$a(b^s a w)a^{-1} = b^{-s} a^2 b^{-1} w =b^{1-s} w,\
a(b^s w)a^{-1} = a b^{s} b^{-1} w =b^{1-s} aw.
$$
These relations imply that
$$
\begin{array}{llll}
aw\sim b^2w, & baw\sim b^3w, & b^2aw\sim w, & b^3aw\sim bw; \\
w\sim b^2aw, & bw\sim b^3aw, & b^2w\sim aw, & b^3w\sim baw; \\
aw\sim bw, & baw\sim  w, & b^2aw\sim b^3w, & b^3aw\sim b^2w; \\
w\sim baw, & bw\sim aw, & b^2w\sim b^3aw, & b^3w\sim b^2aw. \\
\end{array}
$$
Moreover,
$$w(b^s a w)w^{-1} = (ab)^{s} b  w =
\left\{
\begin{array}{ll}
 b^{s+1}w & \mbox{ if } s \equiv 0 \tmod{2}, \\
 b^{s+1}aw & \mbox{ if } s \equiv 1 \tmod{2},
\end{array}
\right.$$
$$w(b^s  w)w^{-1} = (ab)^{s}   w =
\left\{
\begin{array}{ll}
 b^{s}w & \mbox{ if } s \equiv 0 \tmod{2}, \\
 b^{s-2}aw & \mbox{ if } s \equiv 1 \tmod{2},
\end{array}
\right.$$
They give no new conjugacy relations.  Thus, the coset $[T'_{8\cdot 3^q},T'_{8\cdot 3^q}] w$ contains two conjugacy classes: $[w]$ and $[bw]$.

(3) The coset $[T'_{8\cdot 3^q},T'_{8\cdot 3^q}] w^2$.  Note that
$$b(b^s a w^2)b^{-1} = b^{s+1} a a^{-1} w^2 = b^{s+1} w^2,\
b(b^s w^2)b^{-1} = b^{s+1} a^{-1} w^2 = b^{s+3}a w^2,$$
$$a(b^s a w^2)a^{-1} = b^{-s} a^2 (ab)^{-1} w^2 =b^{-s-1}a w^2,\
a(b^s w^2)a^{-1} = a b^{s} (ab)^{-1} w^2 =b^{1-s} w^2.
$$
These relations imply that
$$
\begin{array}{llll}
aw^2\sim bw^2, & baw^2\sim b^2w^2, & b^2aw^2\sim b^3w^2, & b^3aw^2\sim w^2; \\
w^2\sim b^3aw^2, & bw^2\sim aw^2, & b^2w^2\sim baw^2, & b^3w^2\sim b^2aw^2; \\
aw^2\sim b^3aw^2, & baw^2\sim  b^2aw^2, & b^2aw^2\sim baw^2, & b^3aw^2\sim aw^2; \\
w^2\sim bw^2, & bw^2\sim w^2, & b^2w^2\sim b^3w^2, & b^3w^2\sim b^2w^2.
\end{array}
$$
Thus, the coset $[T'_{8\cdot 3^q},T'_{8\cdot 3^q}] w^2$ contains two conjugacy classes: $[w^2]$ and $[b^2w^2]$.
\end{proof}

\begin{Lemma}\label{subgroupT8}
The non-trival subgroups and quotients of $T'_{8\cdot 3^q}$ are listed as follows.
\begin{table}[h]
\begin{center}
\begin{tabular}{|c|l|l|l|}
\hline
  & subgroup type & conjugacy classes  & quotient \\ \hline
 1 & $\bbZ_{3^{q-r}}$  ($0<r<q$ ) &  $\{\zhclass{w^{3^r}}\}$
    &  $T'_{8\cdot 3^r}$ \\ \hline
 2 & $\bbZ_{4\cdot 3^{q-r}}$  ($0<r\le q$)
  &  $\{\zhclass{bw^{3^r}}, \zhclass{b^3w^{3^r}}, \zhclass{aw^{3^r}}, \zhclass{baw^{3^r}}, \zhclass{b^2aw^{3^r}}, \zhclass{b^3aw^{3^r}} \}$
    &  -  \\ \hline
 3 & $\bbZ_{2\cdot 3^{q-r}}$  ($0<r\le q$)
   &  $\{\zhclass{b^2w^{3^r}}\}$ &  $(\bbZ_2\times \bbZ_2)\rtimes \bbZ_{3^r}$  \\ \hline
 4 & $\bbZ_{3^{q}}$   &  $\{\zhclass{w},  \zhclass{b^3w}, \zhclass{baw}, \zhclass{b^2aw}\}$
    &  - \\ \hline
 5 & $\bbZ_{2\cdot 3^{q}}$   &  $\{\zhclass{bw},  \zhclass{aw}, \zhclass{b^2w}, \zhclass{b^3aw}\}$
    &  - \\ \hline
 6  & $D^*_{4\cdot 2}\times \bbZ_{3^{q-r}}$ ($0<r\le q$)   &  $\{\zhclass{b,a, w^{3^r}}\}$
    &  $\bbZ_{3^{r}}$ \\ \hline
\end{tabular}
\end{center}
\caption{subgroups and quotients of $T'_{8\cdot 3^q}$}
\end{table}
\end{Lemma}

\begin{proof}
Note that $\zhclass{w^{s}} = \zhclass{w^{(s, 3^q)}}$. The first three rows come from those in Lemma~\ref{conjT8}, where $3^r=(3^q, 3t)$.

Since both $3t+1$ and $3t+2$ are coprime to  $3^q$, we obtain that
$\zhclass{w^{3t+1}}=\zhclass{w^{3t+2}}= \zhclass{w}$. Thus, all
cyclic subgroups generated by the elements of order $3^q$ listed
in Lemma~\ref{conjT8} are conjugate. We obtain the row 4. This
fact follows from Sylow's theorems. Clearly, these subgroups are not
normal.

Consider the elements of order $2\cdot 3^q$. Note that
$(b^2w)^{6s+1} = b^2w^{6s+1}$, $(b^2w)^{6s+5} = b^2w^{6s+5}$.
Since both $6s+1$ and $6s+5$ are relatively prime to $2\cdot 3^q$,
we have $\zhclass{b^2w} =
\zhclass{b^2w^{6s+1}}=\zhclass{b^2w^{6s+5}}$. Note that
$(b^2w)^{3^q-(6s+2)} = b^2w^{-6s-3}$, $(b^2w)^{3^q-(6s+4)} =
b^2w^{-6s-4}$. Since both $3^q-(6s+2)$ and $3^q-(6s+4)$ are
relatively prime to $2\cdot 3^q$, we have $\zhclass{b^2w} =
\zhclass{b^2w^{-6s-2}}=\zhclass{b^2w^{-6s-4}}$. Since
$b^2w^{-6s-2} = (b^2w^{6s+2})^{-1}$ and  $b^2w^{-6s-4} =
(b^2w^{6s+4})^{-1}$, we obtain that $\zhclass{b^2w} =
\zhclass{b^2w^{6s+2}}=\zhclass{b^2w^{6s+4}}$. Thus, there are four
cyclic subgroups, all conjugate to each other.

Let us consider the non-cyclic subgroup $H$ of $T'_{8\cdot 3^q}$. Note that the abelianization of $T'_{8\cdot 3^q}=\zhclass{w}\cong \bbZ_{3^q}$. The abelianization $H/[H,H]$ will be a subgroup of $\zhclass{w}$. We have the following three cases:

(1) If $H/[H,H]$ is trivial then $H$ lies in the commutator $[T'_{8\cdot 3^q}, T'_{8\cdot 3^q}]=D^*_{4\cdot 2}$.  By Lemma~\ref{subgroupD4n}, any proper subgroup of $D^*_{4\cdot 2}$ is cyclic, and hence $H$ must be the commutator itself. This is the row 6 for $r=q$.

(2) If $H/[H,H]=\zhclass{w}$, then by Lemma~\ref{conjT8} and up to a conjugation, $H$ must contain $w$ or $b^2w$. If $b^2w\in H$, then $(b^2w)^2=w^2\in H$. Since $(2,3^q)=1$, we still have that $w\in H$.  Since $H$ is non-cyclic, $H$ contains an element of the form $xw^k$, where $x$ is a non-trivial element in the commutator. Hence $x= (xw^k)w^{-k}$ also lies in $H$. Note that either $x$ or $x^2$ is $b^2$. If $x=b^2$ is the unique non-trivial element in the commutator, then $H=\zhclass{b^2, w}$ would be cyclic. Thus, there must an element $x$ in $H$ with $x^2=b^2$. Thus, one of the six elements: $x, wxw^{-1}, w^2xw^{-2}, x^3, x^2wxw^{-1}, x^2w^2xw^{-2}$ must be $a$. Note that $\zhclass{a,w}=T'_{8\cdot 3^q}$. Then $H$ must be the whole group.

(3) If $H/[H,H]=\zhclass{w^{3^r}}$ with $0<r<q$, then $H=H'\times \zhclass{w^{3^r}}$ for some subgroup $H'$ of the commutator $D^*_{4\cdot 2}$ since $w^{3^r}$ lies in the center. By Lemma~\ref{subgroupD4n}, any proper subgroup of $D^*_{4\cdot 2}$ is cyclic, and hence $H'$ must be the commutator itself. This is the row 6 for $0<r<q$.
\end{proof}

\begin{Lemma}\label{conjD2q}
The set of conjugacy classes of $D'_{n\cdot 2^q}$ with $2\not| n$ is given as follows.
\begin{table}[h]
\begin{center}
\begin{tabular}{|l|l|l|}
\hline
  representative & order & all elements \\ \hline
  $[w^{2t+1}]$\  ($0\le t<2^{q-1}-1$) & $2^q$
      & $\{u^{s}w^{2t+1}\mid 0\le s < n\}$  \\   \hline
  $[w^{2t}]$   ($ 0\le t<2^{q-1}-1$)  & $\frac{2^q}{(2^q, 2t)}$  &   $\{w^{2t}\}$  \\ \hline
 $[u^sw^{2t}]$\  ($0< s \le \frac{n-1}{2},\ 0\le t<2^{q-1}$)
    & $\frac{n}{(s,n)}\times \frac{2^q}{(2^q, 2t)} $ &  $\{u^sw^{2t}, u^{n-s}w^{2t}\}$  \\ \hline
\end{tabular}
\end{center}
\caption{conjugacy classes of $D'_{n\cdot 2^q}$ with $2\not| n$}
\end{table}
\end{Lemma}
\begin{proof}
Note that $w(u^kw^l)w^{-1} = u^{-k}w^{l}$,  and note that $u(u^kw^l)u^{-1} = u^{k+2}w^{l}$ if $l$ is odd; and $u(u^kw^l)u^{-1} = u^{k}w^{l}$ if $l$ is even. Moreover $u^kw^l = u^{k+n'}w^l$.
\end{proof}

\newpage

\begin{Lemma}\label{subgroupD2q}
The non-trival subgroups and quotients of $D'_{n\cdot 2^q}$  with $2\not| n$  are listed as follows.
\begin{table}[h]
\begin{center}
\begin{tabular}{|c|l|l|l|}
\hline
  & subgroup type & conjugacy classes  & quotient \\ \hline
 1 &  $\bbZ_{2^q}$ &  $\{\zhclass{u^sw}\mid s=0, 1, \ldots, n-1\}$  &   -  \\ \hline
 2 &  $\bbZ_{2^{q-1}}$   &  $\zhclass{w^{2}}$
   &  $\bbZ_n \rtimes \bbZ_2$ \\  \hline
 3 &  $\bbZ_{2^{q-r}}$ \ ($1<r<q$) &  $\zhclass{w^{2^{r}}}$
   &  $D'_{n\cdot 2^{r}}$ \\  \hline
 4 &  $\bbZ_{n}$
 & $\zhclass{u}$
 & $\bbZ_{2^q}$  \\  \hline
 5 &  $\bbZ_{k}$ \ ($k|n,\ 1<k< n$)
 & $\zhclass{u^k}$
 & $D'_{k\cdot 2^{q}}$  \\  \hline
 6 &  $\bbZ_{n\cdot 2^{q-r}}$ \ ($0<r< q$)
 & $\zhclass{u w^{2^r}}$
 & $\bbZ_{2^r}$  \\  \hline
 7 &  $\bbZ_{\frac{n}{k}\cdot 2^{q-1}}$ \ ($k|n,\ 1<k<n$)
 & $\zhclass{u^k w^2}$
 & $\bbZ_k\rtimes  \bbZ_2$  \\  \hline
 8 &  $\bbZ_{\frac{n}{k}\cdot 2^{q-r}}$ \ ($k|n,\ 1<k<n,\ 1<r< q$)
 & $\zhclass{u^k w^{2^r}}$
 & $D'_{k\cdot 2^r}$  \\  \hline
 9 & $D'_{\frac{n}{k}\cdot 2^q}$ \ \ ($k|n, 1<k<n$)&   $\{\zhclass{u^k, u^sw}\mid s =0,1, \ldots, n\}$ & - \\ \hline
\end{tabular}
\end{center}
\caption{subgroups and quotients of $D'_{n\cdot 2^q}$ with $2\not| n$}
\end{table}
\end{Lemma}

\begin{proof}
Consider cyclic subgroups. For the conjugacy classes of the form $[w^{2t+1}]$, since $2t+1$ is coprime with the order $2^q$ of $w$, we have that $\zhclass{w^{2t+1}} = \zhclass{w}$. By Lemma~\ref{conjD2q}, we obtain the first row.

For the conjugacy classes of the form $[w^{2t}]$, we have that $\zhclass{w^{2t}} = \zhclass{u^{2^r}}$ for some $r$ with $0<r<q$. These are the second and third row.

For the conjugacy classes of the form $[u^s]$,  we have that $\zhclass{u^s} = \zhclass{u^k}$ for some $k$ with $k|n$, giving row $4$ and $5$.

For the conjugacy classes of the form $[u^sw^{2t}]$,  we have that $\zhclass{u^sw^{2t}} = \zhclass{u^k w^{2^r)}}\cong \bbZ_{\frac{n}{k}}\times \bbZ_{2^{q-r}}$ for some $k$ and $r$ with $k|n$ and $0<r<q$, which is a normal subgroup. We have the rows from $6$ to $8$, which are respectively the case $k=1$, $k>1$ and $r=1$, and $k,r>1$.

Assume that $H$ is a non-cyclic subgroup of $D'_{n\cdot 2^q}$, $H$ must contains an element of the form $u^sw^{2t+1}$. Up to a conjugation, we may have that $w^{2t+1}\in H$. It follows that $w\in H$ because $w\in \zhclass{w^{2t+1}}$. Since $H$ is not cyclic, $H$ contain some element of the form $u^s$. Let $u^k$ be the element in $H$ such that $k$ is the minimal positive integer. Thus, $H=\zhclass{u^k, w}$. This is the final row. If such a subgroup is normal, then $uwu^{-1}=u^2w$ must lies in $H$, and  hence $u^2\in H$. Note that $k$ is odd. We have that $u^{(k,2)} = u$ lies in $H$. Thus, $H=\zhclass{u, w}$ is the whole group.
\end{proof}

\subsection{Embedding $G$ in $SO(4)$}

By definition of spherical $3$-manifolds, we may regard the fundamental groups of spherical $3$-manifolds as discrete subgroups of the orientation-preserving isometry group $SO(4)$ of $S^3$.

\begin{proposition}\label{inso4} Let $G$ be the fundamental group of a spherical $3$-manifold. If $G$ is not cyclic, then the generators in the presentation of Proposition~\ref{presentation} can be identified as elements in $SO(4)$ as follows.

(1) if $G=\bbZ_m\times D^*_{4n}$,
$$
b=
\left(
  \begin{array}{cccc}
    \cos\frac{\pi}{n} & 0 & 0 & -\sin\frac{\pi}{n} \\
    0 & \cos\frac{\pi}{n} & \sin\frac{\pi}{n} & 0 \\
    0 & -\sin\frac{\pi}{n} & \cos\frac{\pi}{n} & 0 \\
    \sin\frac{\pi}{n} & 0 & 0 & \cos\frac{\pi}{n}
  \end{array}
\right),
\
a=
\left(
  \begin{array}{cc}
    R(\frac{\pi}{2}) & 0 \\
    0 & R(-\frac{\pi}{2}) \\
  \end{array}
\right)
;
$$

(2) if $G=\bbZ_m\times O^*_{48}$,
$$
b=\frac{1}{2}
\left(
\begin{array}{cccc}
 1 & -1 & -1 & -1 \\
 1 & 1 & 1 & -1 \\
 1 & -1 & 1 & 1 \\
 1 & 1 & -1 & 1
\end{array}
\right),
\
a=\frac{1}{\sqrt{2}}
\left(
\begin{array}{cccc}
 0 & 1 & 1 & 0 \\
 -1 & 0 & 0 & 1 \\
 -1 & 0 & 0 & -1 \\
 0 & -1 & 1 & 0
\end{array}
\right);
$$

(3) if $G=\bbZ_m\times I^*_{120}$,
$$
b=\frac{1}{2}
\left(
\begin{array}{cccc}
 1 & -1 & -1 & -1 \\
 1 & 1 & 1 & -1 \\
 1 & -1 & 1 & 1 \\
 1 & 1 & -1 & 1
\end{array}
\right),
\
a=\frac{1}{2}
\left(
  \begin{array}{cc}
    R(-\frac{\pi}{2}) & R(-\frac{2\pi}{5}) \\
    R(\frac{7\pi}{5}) & R(\frac{\pi}{2}) \\
  \end{array}
\right),
$$

(4) if $G=\bbZ_m\times T'_{8\cdot 3^q}$,
$$
a=
\left(
  \begin{array}{cc}
    R(-\frac{\pi}{2}) & 0 \\
    0 & R(\frac{\pi}{2}) \\
  \end{array}
\right),
\
w=\frac{1}{2}
\left(
\begin{array}{cccc}
 -1 & -1 & -1 & -1 \\
 1 & -1 & 1 & -1 \\
 1 & -1 & -1 & 1 \\
 1 & 1 & -1 & -1
\end{array}
\right)
\left(
  \begin{array}{cc}
    R(\frac{2\pi}{3^q}) & 0 \\
    0 & R(\frac{2\pi}{3^q})\\
  \end{array}
\right);$$

(5) if $G=\bbZ_m\times D'_{n\cdot 2^q}$,
$$
u=
\left(
  \begin{array}{cc}
    R(\frac{2\pi}{n}) & 0 \\
    0 & R(-\frac{2\pi}{n}) \\
  \end{array}
\right),
\
w=
\left(
  \begin{array}{cc}
    0 & R(\frac{\pi}{2^{q-1}})  \\
     R(\frac{\pi}{2^{q-1}}) & 0 \\
  \end{array}
\right).
$$

In all cases, the generator of $\bbZ_m$ is given by $v=\left(
  \begin{array}{cc}
    R(\frac{2\pi}{m}) & 0 \\
    0 & R(\frac{2\pi}{m}) \\
  \end{array}
\right)$.
Here, we use $R(\theta)$ to denote the matrix (block)
$
\left(
  \begin{array}{cc}
    \cos \theta & -\sin \theta \\
    \sin \theta & \cos \theta \\
  \end{array}
\right)$.
\end{proposition}

From now on, all groups are subgroups of $SO(4)$. By Olum's results \cite{Olum1,Olum2}, every homomorphism $\psi: G_1\to G_2$ can be realized by a map $f: S^3/G_1\to S^3/G_2$  with $\psi=f_{\pi}$. Moreover, if $f$ and $g$ induce the same homomorphism $\psi$, then $\deg f \equiv \deg g$ mod $|G_2|$. Therefore, we have a well-defined mod $|G_2|$ integer associated with any homomorphism $\psi$ which we denote by $\dg(\psi)\in \bbZ_{|G_2|}$.

It is known that the orbit spaces of the form $S^3/\bbZ_m$ are lens spaces. The homotopy types and homeomorphism type of $3$-dimensional lens spaces depend on the $\bbZ_m$ action, which can be regarded as embeddings of $\bbZ_m$ in $SO(4)$.

\section{Lens paces covering spherical manifolds}

The generalized $3$-dimensional lens space $L(m; r_1, r_2)$ is defined to be the orbit space $S^3/\bbZ_m$, where $\bbZ_m$ is the finite cyclic subgroup of $SO(4)$, which is generated by
$$c=
\left(
  \begin{array}{cc}
    R(\frac{2\pi r_1}{m}) & 0 \\
    0 & R(\frac{2\pi r_2}{m}) \\
  \end{array}
\right).
$$
Here $r_1$ and $r_2$ are relatively prime to $m$, but $m$ is not necessarily prime. If we regard $S^3$ as  the unit sphere of complex space $\mathbb{C}^2$, the element $c$ turns out to be the complex matrix
$$c=
\left(
  \begin{array}{cc}
    e^{\frac{2\pi r_1}{m}i} & 0 \\
    0 & e^{\frac{2\pi r_2}{m}i} \\
  \end{array}
\right).
$$
By identifying the covering transformation group with the fundamental group, we call this element $c$ the {\it standard generator} of $\pi_1(L(m; r_1, r_2))$.

It is known that $L(m; r_1, r_2)$ and $L(m; r'_1, r'_2)$ are homeomorphic if and only if $\{\pm \frac{r_1}{r_2}, \pm \frac{r_2}{r_1}\}$ and $\{\pm \frac{r'_1}{r'_2}, \pm \frac{r'_2}{r'_1}\}$ are the same subset of $\bbZ_m^*$, the multiple group of invertible elements in $\bbZ_m$ (see e.g. \cite[Ch. V]{C}). Usually,  $L(m; 1, r)$ is written as $L(m; r)$. Clearly, $L(m; r_1, r_2) = L(m; 1, r')$ if $r'=r_1[r_2]^{-1}_{(m)}$. Here, $[r_2]^{-1}_{(m)}$ denotes the inverse of the multiplication of $\bbZ_{m}$.

The lens spaces that can cover $L(m; r_1, r_2)$ must be  $L(k; r_1, r_2)$ for some $k|m$.

\begin{example} Consider the group $O^*_{48}$ and $H=\langle b\rangle$ the cyclic group of order $6$ generated by $b$.
The eigenvalues of the matrix $b$ are $\{e^{\frac{\pi i}{3}}, e^{-\frac{\pi i}{3}}, e^{\frac{\pi i}{3}}, e^{-\frac{\pi i}{3}}\}$, meaning either $e^{\frac{\pi i}{3}}$ or $e^{-\frac{\pi i}{3}}$ has multiplicity two. It follows that $b$ is conjugate to
$\left(
  \begin{array}{cc}
    R(\frac{\pi}{3}) & 0 \\
    0 & R(\frac{\pi}{3}) \\
  \end{array}
\right)$.
Thus, the corresponding lens space is $L(6;1,1)$

More generally, consider the group $\mathbb Z_m\times O^*_{48}$ and $H=\langle vb \rangle$, the eigenvalues of  the matrix $vb$ are  $\{e^{\frac{(6+m)\pi i}{3}}, e^{- \frac{(6+m)\pi i}{3}}, e^{\frac{(6-m)\pi i}{3}}, e^{- \frac{(6-m)\pi i}{3}}\}$. Thus,  $vb$ is conjugate to
$\left(
  \begin{array}{cc}
    R(\frac{(6+m)\pi}{3}) & 0 \\
    0 & R(\frac{(6-m)\pi}{3}) \\
  \end{array}
\right)$.
\end{example}

The example above shows in general how one determines the lens space corresponding to a maximal cyclic subgroup. We summarize the results below.

\newpage

\begin{Lemma}\label{alllensspace}
Let $H$ be a maximal cyclic subgroup of some $G$, where $G$ is the fundamental group of some spherical $3$-manifold. Then the lens spaces of the form $S^3/H$ are listed as follows:
\begin{table}[h]
\begin{center}
\begin{tabular}{|c|l|l|l|}
  \hline
  $G$ & $H$  & $|H|$ & lens space \\ \hline
  $\bbZ_m\times D^*_{4n}$ & $\langle v, a\rangle$ &  $4m$ &  $L(4m;\ 4-m, 4+m)$\\ \hline
  $\bbZ_m\times D^*_{4n}$ & $\langle v, ba\rangle$ &  $4m$ &  $L(4m;\ 4-m, 4+m)$\\ \hline
  $\bbZ_m\times D^*_{4n}$ & $\langle v, b \rangle$ &  $2mn$ &  $L(2mn;\ 2n-m, 2n+m)$\\ \hline
  $\bbZ_m\times O^*_{48}$ & $\langle v, b \rangle$ & $6m$ & $L(6m;\ 6-m, 6+m)$ \\ \hline
  $\bbZ_m\times O^*_{48}$ & $\langle v, ba \rangle$ & $8m$ & $L(8m;\ 8-m, 8+m)$ \\ \hline
  $\bbZ_m\times O^*_{48}$ & $\langle v, ba^{-1} \rangle$ & $8m$ & $L(8m;\ 8-3m, 8+3m)$ \\ \hline
  $\bbZ_m\times I^*_{120}$ &  $\langle v, a\rangle$ &  $4m$ &  $L(4m;\ 4-m, 4+m)$ \\ \hline
  $\bbZ_m\times I^*_{120}$ &  $\langle v, b\rangle$ &  $6m$ &  $L(6m;\ 6-m, 6+m)$ \\ \hline
  $\bbZ_m\times I^*_{120}$ &  $\langle v, ba\rangle$ &  $10m$ &  $L(10m;\ m-10, m+10)$ \\ \hline
  $\bbZ_m\times I^*_{120}$ &  $\langle v, baba^{-1}\rangle$ &  $10m$ &  $L(10m;\ 3m-10, 3m+10)$ \\ \hline
$\bbZ_m\times T'_{8\cdot 3^q}$ & $\langle v, bw\rangle,\ \langle v, b^2w^2\rangle$ & $2m3^q$ &
$L(2m3^q;\ 2m+2\cdot 3^q-m3^{q-1},\ 2m+2\cdot 3^q+m3^{q-1})$ \\ \hline
$\bbZ_m\times T'_{8\cdot 3^q}$ & $\langle v, a, w^3\rangle$ & $4m3^{q-1}$ &
  $L(4m3^{q-1};\ 4m+4\cdot3^{q-1} - m3^{q-1},\ 4m+4\cdot3^{q-1} + m3^{q-1})$ \\ \hline
  $\bbZ_m\times D'_{n\cdot 2^q}$ & $\langle v, w\rangle$ & $m 2^q$
   & $L(m2^q;\ m+ 2^{q-1},\ m+ 2^{q-1}+ m 2^{q-1})$ \\ \hline
 $\bbZ_m\times D'_{n\cdot 2^q}$ &  $\langle v, u, w^2 \rangle$ & $mn 2^{q-1}$ &  $L(mn2^{q-1};\  n2^{q-1}+mn+m2^{q-1},\ n2^{q-1}+mn - m2^{q-1})$  \\ \hline
\end{tabular}
\end{center}
\caption{lens spaces that cover $S^3/G$}
\end{table}
\end{Lemma}

\begin{proof} By Lemma~\ref{presentation}, the generator of a cyclic group can be written as a matrix in $SO(4)$. We can compute its eigenvalues for any given matrix. Note that the corresponding lens spaces are completely determined by the eigenvalues of the matrices corresponding to the generators. This proof is a straightforward computation.
\end{proof}

\begin{Lemma}\label{deglenspace}
Let $f: L(m_1; r_{11}, r_{12})\to L(m_2; r_{21}, r_{22})$ be a map between lens spaces,
inducing a homomorphism $f_\pi$ on fundamental group given by $f_\pi(c_1) = c_2^l$, where $c_1$ and $c_2$ are respectively standard generators of $\pi_1(L(m_1; r_{11}, r_{12}))$ and $\pi_1(L(m_2; r_{21}, r_{22}))$. Then the degree of $f$ is
$$\deg (f) = l^2 r_{21}r_{22}\frac{m_1}{m_2}[r_{11}]^{-1}_{(m_1)}[r_{12}]^{-1}_{(m_1)}+km_2 \quad \text{for some }k\in \mathbb Z, $$
where $[*]^{-1}_{(m_1)}$ denotes the inverse of multiplication of $\bbZ_{m_1}$.
\end{Lemma}

\begin{proof}
Let us consider the case when $f_\pi$ is surjective.  Then we have that $m_2|m_1$. We construct a map $\tilde g: S^3\to S^3$ by $\tilde g
\left(
  \begin{array}{c}
    r_1e^{\theta_1 i}\\
    r_2e^{\theta_2 i}
  \end{array}
\right)
=
\left(
  \begin{array}{c}
    r_1e^{k_1\theta_1 i}\\
    r_2e^{k_2\theta_2 i}
  \end{array}
\right).
$
Note that
$$
\left(
  \begin{array}{c}
    r_1e^{\theta_1 i}\\
    r_2e^{\theta_2 i}
  \end{array}
\right)
\stackrel{c_1}{\mapsto}
\left(
  \begin{array}{c}
    r_1 e^{\theta_1 i+ \frac{2\pi r_{11}}{m_1}i}\\
    r_2 e^{\theta_2 i+\frac{2\pi r_{12}}{m_1}i}
  \end{array}
\right)
\stackrel{\tilde g}{\mapsto}
\left(
  \begin{array}{c}
    r_1 e^{k_1\theta_1 i+ \frac{2\pi k_1r_{11}}{m_1}i}\\
    r_2 e^{k_2\theta_2 i+ \frac{2\pi k_2r_{12}}{m_1}i}
  \end{array}
\right),
$$
$$
\left(
  \begin{array}{c}
    r_1e^{\theta_1 i}\\
    r_2e^{\theta_2 i}
  \end{array}
\right)
\stackrel{\tilde g}{\mapsto}
\left(
  \begin{array}{c}
    r_1e^{k_1\theta_1 i}\\
    r_2e^{k_2\theta_2 i}
  \end{array}
\right)
\stackrel{c_2^l}{\mapsto}
\left(
  \begin{array}{c}
    r_1 e^{k_1\theta_1 i + \frac{2\pi r_{21}l}{m_2}i} \\
    r_2 e^{k_2\theta_2 i + \frac{2\pi r_{22}l}{m_2}i} \\
  \end{array}
\right).
$$
In order to obtain that $\tilde g c_1 = c_2^l \tilde g$, it is sufficient to assume that
$\frac{r_{11}k_1}{m_1}- \frac{r_{21}l}{m_2}$ and $\frac{r_{12}k_2}{m_1}-\frac{r_{22}l}{m_2}$ are both integers. Thus, we may pick $k_1=lr_{21}\frac{m_1}{m_2}[r_{11}]^{-1}_{(m_1)}$  and $k_2=lr_{22}\frac{m_1}{m_2}[r_{12}]^{-1}_{(m_1)}$. It is obvious that $\tilde g$ has degree $$\deg (\tilde g) = k_1k_2 = l^2 r_{21}r_{22}(\frac{m_1}{m_2})^2[r_{11}]^{-1}_{(m_1)}[r_{12}]^{-1}_{(m_1)}. $$

The equality $\tilde g c_1 = c_2^l \tilde g$ also implies that $\tilde g$ induces a  map $g: L(m_1; r_{11}, r_{12})\to L(m_2; r_{21}, r_{22})$. Since $\tilde g$ is a lifting of $g$ with respect to the universal covering $S^3$ of both $L(m_1; r_{11}, r_{12})$ and $L(m_2; r_{21}, r_{22})$, we have that $\deg(\tilde g) m_2 = m_1\deg(g)$. Here $m_1$ and $m_2$ are respectively the degrees of the covering maps $S^3\to L(m_1; r_{11}, r_{12})$ and $S^3\to L(m_2; r_{21}, r_{22})$. Thus,
$$\deg (g) =  l^2 r_{21}r_{22}\frac{m_1}{m_2}[r_{11}]^{-1}_{(m_1)}[r_{12}]^{-1}_{(m_1)}. $$

Note that $f$ and $g$ have the same induced homomorphism on $\pi_1$, which is given by $c_1\mapsto c_2^l$. By Olum results \cite{Olum1,Olum2}, we have that $\deg(f) \equiv \deg(g) \tmod{m_2}$. Thus, we give the proof of $\deg(f)$ when $f_\pi$ is surjective.

Now, we turn to the general case. The image $\im f_\pi$ of $f_\pi$ is the cyclic group generated by $c_2^{(l,m_2)}$, which is isomorphic to $\bbZ_{\frac{m_2}{(l,m_2)}}$. Note that the covering space of $L(m_2, r_{21}, r_{22})$ corresponding to the group $\langle c_2^{(l,m_2)}\rangle$ is $L(\frac{m_2}{(l,m_2)}, r_{21}, r_{22})$. There must be a lifting $\tilde f$ of $f$ satisfying the following commutative diagram:
$$
\xymatrix{
       &  L(\frac{m_2}{(l,m_2)}; r_{21}, r_{22}) \ar[d]^{p}  \\
 L(m_1; r_{11}, r_{12}) \ar[ur]^{\tilde f} \ar[r]^{f} &  L(m_2; r_{21}, r_{22})
}
$$
Note that $c_2^l = (c_2^{(l,m_2)})^{\frac{l}{(l,m_2)}}$. Apply our previous argument to $\tilde f$ which induces a surjection on $\pi_1$, we have that
$$\deg(\tilde f)
= (\frac{l}{(l,m_2)})^2 r_{21}r_{22}\frac{m_1}{m_2/(l,m_2)}[r_{11}]^{-1}_{(m_1)}[r_{12}]^{-1}_{(m_1)}
=
\frac{l^2}{(l,m_2)} r_{21}r_{22}\frac{m_1}{m_2}[r_{11}]^{-1}_{(m_1)}[r_{12}]^{-1}_{(m_1)}.
$$
Since $\deg(f) = \deg(p)\deg(\tilde f)= (l, m_2) \deg(\tilde f)$, we obtain our conclusion.
\end{proof}

It should be mentioned that the existence of the homomorphism $c_1\to c_2^l$ forces that $m_1 l\equiv 0 \tmod{m_2}$. In other word, $\frac{lm_1}{m_2}$ must be an integer.

\begin{theorem}\label{degreebetweenlensspaces}
The set of all degrees $D(L(m_1; r_{11}, r_{12}), L(m_2; r_{21}, r_{22}))$ of maps from $L(m_1; r_{11}, r_{12})$ to $L(m_2; r_{21}, r_{22})$ is:
$$
\{ j^2 r_{21}r_{22}\frac{m_1m_2}{(m_1,m_2)^2}[r_{11}]^{-1}_{(m_1)}[r_{12}]^{-1}_{(m_1)} +km_2 \mid k\in \mathbb Z, j=0,1,\ldots, (m_1, m_2)-1\}. $$
\end{theorem}

\begin{proof}
For any map $f$ between these two lens spaces, the homomorphism $f_\pi$ on fundamental groups is given by $c_1\to c_2^l$. Note that $m_1l \equiv 0 \tmod{m_2}$. It follows that $l=j \frac{m_2}{(m_1, m_2)}$ for   $j=0,1,\ldots, (m_1, m_2)-1$. This theorem holds by Lemma~\ref{deglenspace}.
\end{proof}

\begin{remark}
When we write a lens space $L(m; r_1, r_2)$, its orientation is specified, inheriting the natural orientation of $S^3$. Thus, by above theorem $D(L(3; 1,1), L(3; 1,2)) = \{3k,-1+3k \mid k\in \bbZ\}$, although $L(3; 1,1)$ and $L(3; 1,2)$ are homeomorphic.
\end{remark}

\section{Mapping degrees}

In this section, we determine all mapping degrees between spherical $3$-manifolds. Suppose $f:M\to N$ is a map between two such manifolds with $G_1=\pi_1(M)$ and $G_2=\pi_1(N)$. If $G_1=\mathbb Z_m$ then the degree is already given in the last section. For the other cases of $G_1$, we only need to consider the case when the induced homomorphism $f_{\pi}$ is surjective.

\begin{lemma}\label{degfromself}
Let $\psi: G_1\to G_2$ be a homomorphism and $G_*$ the quotient
$\frac{G_1}{ker(\psi)}.$   Then
$$\dg(\psi) = \dg(\xi)\dg(\varphi)\frac{|G_2|}{|G_*|},$$
where $\varphi: G_1 \to G_*$ is the projection and $\xi: G_*\to \im \psi$
is the isomorphism given by $\varphi(x)\mapsto \psi(x)$.
\end{lemma}

\begin{proof}
Let us consider the following three maps:\\
i) $ f: S^{2k+1}/G_1\to S^{2k+1}/G_*$ be a map such that $f_\pi=\varphi$\\
ii) $h: S^{2k+1}/G_* \to S^{2k+1}/\im \psi$  such that $h_\pi=\xi$\\
iii) $p_2:  S^{2k+1}/\im \psi \to  S^{2k+1}/G_2$.

 It is easy to see that the homomorphism $\psi$ is $p_2\circ \xi\circ \varphi$, the composite of the three homomorphisms. Since the degree of  the covering of $p_2$ is  $\frac{|G_2|}{|G_*|}$
the  formula follows. 
\end{proof}

By using Lemma \ref{degfromself}, since the degree of endomorphisms are almost listed in \cite{Du}, we can compute the degree $\dg(\psi)$, where $\psi: G_1\to G_2$ is a surjection and $G_2$ is a subgroup of $G_1$.

\begin{Lemma}\label{selfdegreeD4n}
Each endomorphism of $D^*_{4n}$ is conjugate to a composition of some endomorphisms in following table.
\begin{table}[h]
\begin{center}
\begin{tabular}{|c|lr|l|l|l|l|l|l|}
\hline
$\psi$ & $\ker \psi$ & & $\psi(b)$ & $\psi(a)$ &   $\im \psi$ &  $\dg(\psi)$
 \\ \hline
$\psi_1$ & $ D^*_{4n}$ &  &  $1$ & $1$   &  $1$ & $0$   \\ \hline
$\psi_2$ & $\zhclass{b^2,a}\cong D^*_{4\cdot \frac{n}{2}}$\ & ($2|\frac{n}{2}$)
  &  $b^n$ & $1$   &  $\zhclass{b^n}\cong \bbZ_2$ & $2n$  \\ \hline
 & $\zhclass{b^2,a}\cong D'_{\frac{n}{2}\cdot 2^2}$\ & ($2\not|\frac{n}{2}$)
  &  $b^n$ & $1$    &  $\zhclass{b^n}\cong \bbZ_2$   & $0$  \\ \hline
$\psi_3$ & $\zhclass{b}\cong \bbZ_{2n}$ &
  & $1$   &  $b^n$  &  $\zhclass{b^n}\cong \bbZ_{2}$  & $0$   \\ \hline
$\psi_{4,k}$ & $\zhclass{b^{\frac{2n}{k}}}\cong \bbZ_{k}$\ & $(k|n, 2\not|k)$
 & $b^{k}$   &  $a$  &  $\zhclass{b^{k}, a}\cong D^*_{4\cdot \frac{n}{k}}$
 & $k^2$   \\ \hline
$\psi_{5,l}$ & $1$ & &  $b^l$\ $(2n,l)=1$  &  $a$ &  $D^*_{4n}$   & $l^2$ \\ \hline
$\psi_6$ & $1$ & &  $b$  &  $ba$ &  $D^*_{4n}$   & $1$   \\ \hline
$\psi_7$ & $\zhclass{b^4}\cong \bbZ_{\frac{n}{2}}$ & ($2\not|\frac{n}{2}$)
 &  $b^{\frac{n}{2}}a$  &  $a$ & $\zhclass{b^{\frac{n}{2}}, a}\cong D^*_{4\cdot 2}$
 & $\frac{n^2}{4}$ \\ \hline
$\psi_8$ & $\zhclass{b^4}\cong \bbZ_{\frac{n}{2}}$ & ($2\not|\frac{n}{2}$)
 &  $a$  &  $b^{\frac{n}{2}}$ & $\zhclass{b^{\frac{n}{2}}, a}\cong D^*_{4\cdot 2}$
 & $\frac{n^2}{4}$ \\ \hline

\end{tabular}
\end{center}
\caption{endomorphisms of $D^*_{4n}$}
\end{table}
\end{Lemma}

\begin{proof}
Let $\psi$ be an endomorphism on $D^*_{4n}$.
By Lemma~\ref{conjD4n}, up to a conjugation, we may assume $\psi(a)=1, a, ba, b^k$.

(1) $\psi(a)=1$.

(1.1) $\psi(b)=b^s$. Then $\psi(a^2)=1$, $\psi(b^n)=b^{ns}$ and $\psi((ab)^2)=b^{2s}$. The relation $a^2=b^n=(ab)^2$ yields that $0 \equiv ns\equiv 2s \tmod{2n}$. Thus, $s\equiv 0,n \tmod{2n}$. Both are in the table ($\psi_1$ and $\psi_2$).

(1.2) $\psi(b)=b^sa$. Note that $\psi((ab)^2)=b^sab^sa=a^2\ne 1 =\psi(a^2)$. Thus, this case is impossible.

(2) $\psi(a)=a$.

(2.1) $\psi(b)=b^s$. Then $\psi(a^2)=a^2=b^n$, $\psi(b^n)=b^{ns}$ and $\psi((ab)^2)=ab^{s}ab^s=a^2$. The relation $a^2=b^n=(ab)^2$ yields that $ns\equiv n \tmod{2n}$. It follows that $s$ is odd. Note that $s=s's''$, where $s'|2n$ and $(2n, s'')=1$. We have that $\psi= \psi_{5, s''} \psi_{4,s'}$.

(2.2) $\psi(b)=b^sa$. Then $\psi(a^2)=a^2=b^n$, $\psi(b^n)=(b^sab^sa)^{\frac{n}{2}}=a^n=b^{n^2/2}$ and $\psi((ab)^2)=ab^saab^sa=b^{-2s}$. The relation $a^2=b^n=(ab)^2$ yields that $n\equiv n^2/2 \equiv -2s \tmod{2n}$. It follows that $\frac{n}{2}$ is odd and $s=\frac{n}{2}, \frac{3n}{2}$. Since $ab^{\frac{n}{2}}aa^{-1}=b^{-\frac{n}{2}}a =b^{\frac{3n}{2}}a$, we can consider only the case $\psi(b)=b^{\frac{n}{2}}a$, which is $\psi_7$.

(3) $\psi(a)=ba$. Note that there is an isomorphism $\psi_6$ with $b\mapsto b$ and $a\mapsto ba$ in this case. By composing with such an isomorphism, this case is the same as case (2).

(4) $\psi(a)=b^k$. Since $a^4=b^{2n}=1$, we must have that $4k\equiv 0 \tmod{2n}$. Up to conjugation by $a$, we may assume that $k=\frac{n}{2}$ or $n$.

(4.1) $\psi(a)=b^{\frac{n}{2}}$.

(4.1.1)  $\psi(b)=b^s$. Then $\psi(a^2)=b^n$, $\psi(b^n)=b^{ns}$ and $\psi((ab)^2)=(b^{\frac{n}{2}}b^s)^2=b^{n+2s}$. The relation $a^2=b^n=(ab)^2$ yields that $n \equiv ns\equiv n +2s \tmod{2n}$, which implies that $s$ is odd and $s \equiv 0 \tmod{n}$. This case is impossible because $n$ is even.

(4.1.2)  $\psi(b)=b^sa$. Then $\psi(a^2)=b^n$, $\psi(b^n)=(b^sab^sa)^{\frac{n}{2}}=a^n$ and $\psi((ab)^2)=(b^{\frac{n}{2}}b^sa)^2=a^2$. The relation $a^2=b^n=(ab)^2$ yields that $n \equiv 2 \tmod{4}$. By Lemma~\ref{conjD4n}, up to conjugation, there are two possibilities: either $\psi(b)=a$ or $\psi(b)=ba$. They are respectively $\psi_8$ and $\psi_6\psi_8$.

(4.2) $\psi(a)=b^n$.

(4.2.1)  $\psi(b)=b^s$. Then $\psi(a^2)=1$, $\psi(b^n)=b^{ns}$ and $\psi((ab)^2)=(b^{n}b^s)^2=b^{2n+2s}$. The relation $a^2=b^n=(ab)^2$ yields that $0 \equiv ns\equiv 2n +2s \tmod{2n}$, which implies that $s=0$ or $n$. They are respectively $\psi_3$ and $\psi_2\psi_6$.

(4.2.2)  $\psi(b)=b^sa$. Note that $\psi((ab)^2)=(b^nb^sa)^2=a^2\ne 1=\psi(a^2)$. Thus, this case is also impossible.

Now we turn to the degree.

Consider the endomorphism $\psi_6$ with $\psi_6(b)=b$ and $\psi_6(a)=ba$. A direct computation shows that a linear map on $\mathbb R^4$ determined by the following matrix induces a self-map on $S^3/D^*_{4n}$, whose induced endomorphism on $\pi_1$ is exactly $\psi'$.
$$
\left(
\begin{array}{cccc}
 1 & 1 & 1 & 1 \\
 -\cos \frac{\pi }{n}-\sin \frac{\pi }{n} &
   \cos \frac{\pi }{n}-\sin \frac{\pi }{n} &
   \cos \frac{\pi }{n}+\sin \frac{\pi }{n} &
   \sin \frac{\pi }{n}-\cos \frac{\pi }{n} \\
 \sin \frac{\pi }{n}-\cos \frac{\pi }{n} &
   -\cos \frac{\pi }{n}-\sin \frac{\pi }{n} &
   \cos \frac{\pi }{n}-\sin \frac{\pi }{n} &
   \cos \frac{\pi }{n}+\sin \frac{\pi }{n} \\
 -1 & 1 & -1 & 1
\end{array}
\right)
$$
Since the determinant of the above matrix is $16$, the degree of the corresponding self-map must be $1$. Thus, we obtain that $\dg(\psi_6)=1$.

Let us write $\psi'_2$ and $\psi'_3$ for two homomorphisms from $D^*_{4n}$ to $\bbZ_2=\zhclass{c}$, defined by $\psi'_2(b)=c$, $\psi'_2(a)=1$,  $\psi'_3(b)=1$ and $\psi'_3(a)=c$. They are two non-zero elements in $H^1(S^3/D^*_{4n}, \bbZ_2) = \bbZ_2\oplus \bbZ_2$. The other non-zero element is given by $\psi'_2\psi_6$. If $2|\frac{n}{2}$, we know from \cite[Theorem 2.2]{TZ} that the cube of one element of these three is zero and the cube of the other two are both $1$. Since $H^*(S^3/\bbZ_2)=\bbZ_2[c^*]$, we obtain that $\dg (\varphi) = \varphi^*(c^*)^3[S^3/D^*_{4n}]$ for any $\varphi: D^*_{4n}\to \bbZ_2$, where $[S^3/D^*_{4n}]\in H_3(S^3/D^*_{4n}, \bbZ_2)$ is the fundamental class. Note that $\dg \psi_6 =1 \tmod{4n}$. It follows that $\dg \psi'_2\psi_6 = \dg \psi'_2 =1$ and hence $\dg \psi'_3=0$. These mean that $\dg \psi_2=2n$ and $\dg \psi_3=0$.  Recall from \cite[Theorem 2.2]{TZ} again that the cube of any element in $H^1(S^3/D^*_{4n}, \bbZ_2)$ vanishes if $2\not|\frac{n}{2}$. Thus, $\dg\psi_2 = \dg\psi_3 =0$.

For $\psi_7$ with $b\mapsto b^{\frac{n}{2}}a$ and $a\mapsto a$. It follows that $b^{\frac{n}{2}}\mapsto  (b^{\frac{n}{2}}a)^{\frac{n}{2}}$ which is $b^{\frac{n}{2}}a$  if $n\equiv 2 \tmod{8}$, is $b^{\frac{3n}{2}}a$ if $n\equiv 6 \tmod{8}$. Note that $\zhclass{b^{\frac{n}{2}}, a}=Q_8$. We may identify $a$ as $i$, $b^{\frac{n}{2}}$ as $j$ in quaternion number. Thus, $\psi_7|_{\zhclass{b^{\frac{n}{2}}, a}}(i)=i$ and  $\psi_7|_{\zhclass{b^{\frac{n}{2}}, a}}(j)=\pm k$. Hence, $\psi_7|_{\zhclass{b^{\frac{n}{2}}, a}}$ is an isomorphism. By \cite{HKWZ}, we  have that $\dg(\psi_7|_{\zhclass{b^{\frac{n}{2}}, a}})=1$. Clearly, $\dg(\psi_7|_{\zhclass{b^4}})=0$. Hence, $\dg(\psi_7)=\frac{n^2}{4}$.

For $\psi_8$ with $b\mapsto a$ and $a\mapsto b^{\frac{n}{2}}$. Similar to the case above, $\psi_8|_{\zhclass{b^{\frac{n}{2}}, a}}$ is an isomorphism on $Q_8$. By \cite{HKWZ}, we have that $\dg(\psi_8|_{\zhclass{b^{\frac{n}{2}}, a}})=1$. Clearly, $\dg(\psi_8|_{\zhclass{b^4}})=0$. Hence, $\dg(\psi_8)=\frac{n^2}{4}$.

The degrees for $\psi_{4,k}$ and $\psi_{5,l}$ follow from \cite[p.1248, Case II]{Du}.
\end{proof}

It should be mentioned that $\psi_3 =\psi_8\psi_2$, if $2\not| \frac{n}{2}$. The degree for $\psi_2, \psi_7, \psi_8$ seem to be missing in \cite{Du}.

\begin{corollary}
If $n>2$, then the group $Out(D^*_{4n})\cong \bbZ_2\times \bbZ_{n}^*$ is generated by $\eta_1, \eta_2$, where $\eta_1(b)=b$, $\eta_1(a)=ba$ and $\eta_2(b)=b^s$, $\eta_2(a)=a$. Here $s$ is an element generating
$\bbZ_{n}^*\cong \bbZ_{\varphi(n)}$.

The group $Out(D^*_{4\cdot 2})\cong S_3$ is generated by $\eta_3, \eta_4$, where $\eta_3(b)=ba$, $\eta_3(a)=a$ and $\eta_4(b)=a$, $\eta_2(a)=b$.
\end{corollary}

\begin{theorem}
Let $\varphi$ be a surjective non-trivial homomorphism from $D^*_{4n}$ to the fundamental group $G$ of some $3$-dimensional spherical manifold $S^3/G$. Then $\varphi = \eta'\psi \eta''$, where $\eta''\in Aut(D^*_{4n})$ and $\eta'\in Aut(G)$. Moreover, all possible $G$ and $\psi$ are listed as follows:
\begin{table}[h]
\begin{center}
\begin{tabular}{|l|l|l|l|l|l|l|}
\hline
$G$ &  $S^3/G$ & $\psi(b)$ & $\psi(a)$ &    $\dg(\psi)$ \\ \hline
$D^*_{4m} =\zhclass{ \bar b, \bar a \mid \bar a^2 = \bar b^n= (\bar a\bar b)^2, \bar a^4 =1}$
 & $S^3/D^*_{4m}$ & $\bar b$ & $\bar a$ &  $\frac{n}{m}$ \\ \hline
$\bbZ_2 = \zhclass{c\mid c^2}$ & $L(2;1,1)$ & $c$ & $1$ &  $1+\frac{n}{2}$ \\ \hline
$\bbZ_2 = \zhclass{c\mid c^2}$ & $L(2;1,1)$ & $1$ & $c$ &  $0$ \\ \hline
\end{tabular}
\end{center}
\caption{domain $D^*_{4n}$}
\end{table}
\end{theorem}

\begin{proof}
The first row comes from the degree of $\psi_{4,k}$ in Lemma~\ref{selfdegreeD4n}, where $m=\frac{n}{k}$. The other two come from the degrees of $\psi_2$ and $\psi_3$ there.
\end{proof}

For $O^*_{48}$, we know that the only non-trivial proper quotient is $\mathbb Z_2$. Hence,

\begin{theorem}
Let $\varphi$ be a non-trivial surjective homomorphism from $O^*_{48}$ to $G$ induced by a map from $S^3/O^*_{48}$ to some $3$-dimensional spherical manifold $S^3/G$. Then $\varphi$ is conjugate to one of the following.
\begin{center}
\begin{tabular}{|l|l|l|l|l|l|l|}
\hline
$G$ & $S^3/G$ & $\psi(b)$ & $\psi(a)$ &    $\dg(\psi)$ \\ \hline
$O^*_{48}$ & $S^3/O^*_{48}$
 & $b$ & $a$ & $1$ \\ \hline
$O^*_{48}$ & $S^3/O^*_{48}$
 & $b$ & $a^{-1}$ & $25$ \\ \hline
$\bbZ_{2}$  & $L(2; 1, 1)$
 & $1$ & $c$ & $1$ \\ \hline
\end{tabular}
\end{center}
\end{theorem}

\begin{proof}
The first two rows comes from \cite{HKWZ}. Let $f: S^3/O*_{48} \to L(2; 1, 1)$ be a map, inducing a homomorphism indicated in third row. By \cite[Theorem 4.13]{TZ}, $H^3(S^3/O^*_{48}, \bbZ_2)$
is generated by $\beta_1^3$ for the unique non-zero element $\beta_1\in H^1(S^3/O^*_{48}, \bbZ_2)$. It follows that the $\deg(f) \equiv  1 \tmod{2}$ because $f^*(c^3) =\beta_1^3$. We obtain the third row.
\end{proof}

We conclude that $D(O^*_{48})=\{m+48k\mid m=0,1, 24, 25, k\in \mathbb Z\}$. Note that the values $24 + 48k, k\in \mathbb Z$ were missing in \cite{Du}.
For $I^*_{120}$ the only possible proper quotient is the trivial group.

\begin{theorem}
Let $\varphi$ be a non-trivial surjective homomorphism from $I^*_{120}$ to $G$ induced by a map from $S^3/I^*_{120}$ to some $3$-dimensional spherical manifold $S^3/G$, then $\varphi$ is conjugate to either the identity  or the automorphism $\eta$ of $I^*_{120}$ determined by  $a\mapsto a^{-1}$, $b\mapsto b^{-1}abab^{-1}a$, where $G$ is also $I^*_{120}$. Here, $\dg(\eta) \equiv 49 \tmod{120}$.
\end{theorem}

\begin{Lemma}\label{selfdegreeT'}
Each endomorphism of $T'_{8\cdot 3^q}$ is conjugate to a composition of some endomorphisms in the following table
\begin{table}[h]
\begin{center}
\begin{tabular}{|c|lr|l|l|l|l|l|l|}
\hline
$\psi$ & $\ker \psi$ & & $\psi(a)$ & $\psi(w)$ &   $\im \psi$ &  $\dg(\psi)$
 \\ \hline
$\psi_1$ & $1$ &  &  $b^{-1}$ & $w^2$   &  $T'_{8\cdot 3^q}$ & $4-3^{2q+1}$ \\ \hline
$\psi_{2}$ & $\zhclass{b, a, w^{3^{q-1}}}\cong D^*_{4\cdot 2}\times \bbZ_{3}$ &
 &  $1$ & $w^{3}$   &  $\zhclass{w^{3}}\cong \bbZ_{3^{q-1}}$ &
 $9-3^{2q+2}$   \\ \hline
$\psi_3$ & $\zhclass{b, a}\cong D^*_{4\cdot 2}$ &
  &  $1$ & $w$   &  $\zhclass{w}\cong \bbZ_{3^{q}}$ &
 $1-3^{2q}$   \\ \hline
$\psi_4$ & $T'_{8\cdot 3^q}$ &  &  $1$ & $1$   &  $1$ & $0$   \\ \hline
\end{tabular}
\end{center}
\caption{endomorphisms of $T'_{8\cdot 3^q}$}
\end{table}
\end{Lemma}

\begin{proof}
By Lemma~\ref{subgroupT8}, the non-trivial normal subgroups of $T'_{8\cdot 3^q}$ are $\bbZ_{3^{q-r}}=\zhclass{w^{3^{r}}}$, $\bbZ_{2\cdot 3^{q-r}}=\zhclass{b^2w^{3^{r}}}$, or $D^*_{4\cdot 2}\times \bbZ_{3^{q-r}}=\zhclass{b,a, w^{3^{r}}}$ with $0<r\le q$. Since neither $T'_{8\cdot 3^q}/\zhclass{w^{3^{r}}}$ nor $T'_{8\cdot 3^q}/\zhclass{b^2w^{3^{r}}}$ is a subgroup of $T'_{8\cdot 3^q}$ except for the trivial case $T'_{8\cdot 3^q}/\zhclass{w^{3^{q}}} = T'_{8\cdot 3^q}$, any endomorphism on $T'_{8\cdot 3^q}$ has kernel $1$, $D^*_{4\cdot 2}\times \bbZ_{3^{q-r}}$, or $T'_{8\cdot 3^q}$.

(1) $\ker \psi=1$, i.e. $\psi$ is an automorphism. Note that $\psi(w)$ has order $3^q$. By Lemma~\ref{conjT8}, we may assume that $\psi(w)=w^{3t+1}$ or $w^{3t+2}$. Since $a$ has order $4$, $\psi(a)$ has also order $4$, and hence $\psi(a)\in \{b, a, ba, b^{-1}, b^2a, b^{-1}a\}$. Note that $waw^{-1}=b$, $wbw^{-1}=ab=b^{-1}a$, $wb^2aw^{-1}=b^{-1}$ and  $wb^{-1}w^{-1}=ba$. Up to a conjugation, we may also assume that $\psi(a)=a$ or $b^{-1}$. There four subcases:

(1.1) $\psi(w)=w^{3t+1}$ and $\psi(a)=a$. It follows that $\psi(b)=\psi(waw^{-1})=\psi(w)\psi(a)\psi(w^{-1})=b$. From number theory, we know that the set $\bbZ^*_{3^q}$ of all multiple invertible elements in $\bbZ_{3^q}$ is a cyclic group of order $3^q-3^{q-1}$, i.e. $\bbZ^*_{3^q}\cong \bbZ_{3^q-3^{q-1}}$. Moreover, $2$ is a generator of $\bbZ^*_{3^q}$. Thus, $\psi = \psi_1^r$, where $2^r\equiv 3t+1\tmod{3^q}$.

(1.2) $\psi(w)=w^{3t+1}$ and $\psi(a)=b^{-1}$. It follows that $\psi(b)=\psi(waw^{-1})=\psi(w)\psi(a)\psi(w^{-1})=ba$. Hence, $\psi(wbw^{-1}) = w^{3t+1}baw^{-(3t+1)} = abb=b^2a$, but $\psi(ab)=b^{-1}ba=a$. Thus, this is an impossible case.

(1.3) $\psi(w)=w^{3t+2}$ and $\psi(a)=a$. It follows that $\psi(b)=\psi(waw^{-1})=\psi(w)\psi(a)\psi(w^{-1})=b^{-1}a$. Hence, $\psi(wbw^{-1}) = w^{3t+2}b^{-1}aw^{-(3t+2)} = a^{-1}ab=b$, but $\psi(ab)=a b^{-1}a=ba^2=b^3$. Thus, this is an impossible case.

(1.4) $\psi(w)=w^{3t+2}$ and $\psi(a)=b^{-1}$. It follows that $\psi(b)=\psi(waw^{-1})=\psi(w)\psi(a)\psi(w^{-1})=b^2a$. Here, $\psi = \psi_1^k$, where $2^k\equiv 3t+2\tmod{3^q}$.

(2) $\ker \psi = D^*_{4\cdot 2}\times \bbZ_{3^{q-r}}$ with $0<r<q$. Then $\psi(a)=\psi(b)=1$.  Since $\psi(w^{3^r})=1$, the order of $\psi(w)$ is a divisor of $3^r$. By Lemma~\ref{conjT8}, up to a conjugation, we may also assume that $\psi(w) = w^{3t}$. Note that $2$ is a generator of $\bbZ^*_{3^q}$. There is some $s$ such that $3t\cdot 2^s\equiv 3^r\tmod{3^q}$. Thus, $\psi_1^s\psi(w)=w^{3^r}$. It follows that $\psi=\psi_1^{-s}\psi_{2}^r$ if $r>0$; $\psi=\psi_2^{-s}\psi_3$ ir $r=0$. The degree of $\psi_{2}, \psi_3$ were given in \cite[p.1250, Case VI]{Du}.

(3) $\ker\psi = T'_{8\cdot 3^q}$. Clearly, $\psi=\psi_4$.

Consider the degree of $\psi_1$. Since $\psi_1|_{\zhclass{b,a}}$ is an automorphism on $Q_8$, we have that $\dg(\psi_1|_{\zhclass{b,a}}) \equiv 1 \tmod{8}$ by \cite{HKWZ}. Thus, $\dg(\psi_1)\equiv 1 \tmod{8}$. On the other hand,
By Lemma~\ref{deglenspace}, we have that $\dg(\psi_1|_{\zhclass{w}}) \equiv 2^2 \tmod{3^q}$. We obtain that $\dg(\psi_1)\equiv  4-3^{2q+1} \tmod{8\cdot 3^q}$.
\end{proof}

\begin{corollary}
The group $Out(T'_{8\cdot 3^q}) \cong \bbZ_{3^q-3^{q-1}}$ is generated by $\eta$, which is defined by  $\eta(a) = b^{-1}$ and $\eta(w) = w^2$.
\end{corollary}

\begin{theorem}
Let $\varphi$ be a surjective non-trivial homomorphism from $T'_{8\cdot 3^q}$ to $G$ induced by a map from $S^3/T'_{8\cdot 3^q}$ to some $3$-dimensional spherical manifold $S^3/G$. Then $\varphi = \eta'\psi \eta''$, where $\eta''\in Aut(T'_{8\cdot 3^q})$ and $\eta'\in Aut(G)$. Moreover, all possible $G$ and $\psi$ are listed as follows:
\begin{center}
\begin{tabular}{|l|l|l|l|l|l|l|}
\hline
$G$ & $S^3/G$ & $\psi(a)$ & $\psi(w)$ &    $\dg(\psi)$ \\ \hline
$T'_{8\cdot 3^r}$ & $S^3/T'_{8\cdot 3^r}$
 & $\bar a$ & $\bar w$ &  $3^{q-r}$ \\ \hline
$\bbZ_{3^{r}}$ & $L(3^{q-r}; 2-3^{q-1}, 2+3^{q-1})$
 & $1$ & $c$ &  $\frac{1}{8}\times (3-3^{2q+1})^r$ \\ \hline
\end{tabular}
\end{center}
\end{theorem}

\begin{proof}

(1) $\ker \psi = \zhclass{w^{3^r}}$, the homomorphism $\psi: T'_{8\cdot 3^q}\to T'_{8\cdot 3^r}$ can be given by $\psi(b) = \bar b$, $\psi(a) = \bar a$ $\psi(w) = \bar w$. By Lemma~\ref{alllensspace},   $S^3/\zhclass{w}$ and $S^3/\zhclass{\bar w}$ are respectively Lens spaces $L(3^q; 2-3^q, 2+3^q)$ and $L(3^r; 2-3^r, 2+3^r)$. By Lemma~\ref{deglenspace},
$$\begin{array}{rcl}
\dg \psi|_{\zhclass{w}}
 & = & (2-3^{r-1})(2+3^{r-1})3^{q-r}[2-3^{q-1}]^{-1}_{(3^q)}[2+3^{q-1}]^{-1}_{(3^q)}\\
 & = & 3^{q-r}(2-3^{r-1})(2+3^{r-1})\frac{1-3^{q-1}}{2}\frac{1-3^{q-1}}{2}\\
 & \equiv & 3^{q-r} \tmod{3^r}.
\end{array}
$$
Since $\psi|_{\zhclass{b, a}}$ is an isomorphism, $\dg \psi|_{\zhclass{b, a}}\equiv 1 \tmod{8}$. It follows that
$$
\begin{array}{rcl}
\dg \psi
& = & 3^r(4-3^{r+1})[3^{q-r}]^{-1}_{(8)}+3^{q-r}(1-4\cdot 3^r+3^{2r+1})\\
& = & 3^r(4-3^{r+1})(4-3^{q-r+1})+3^{q-r}(1-4\cdot 3^r+3^{2r+1})\\
& \equiv & 3^{q-r} \tmod{8\cdot 3^r}
\end{array}
$$

(2) $\ker \psi =\bbZ_{2\cdot 3^{q-r}}$. Since its quotient $(\bbZ_2\times\bbZ_2)\rtimes \bbZ_{3^r}$ is not a subgroup in our list, this is an impossible case.

(3) $\ker \psi = D^*_{4\cdot 2}\times \bbZ_{3^{q-r}}$ ($0<r\le q$). Then $G=Im \psi =\bbZ_{3^r}=\zhclass{c\mid c^{3^r}}$. The homomorphism $\psi: T'_{8\cdot 3^r}\to \bbZ_{3^r}$ can be given by $\psi(a) = 1$ and $\psi(w) =c$. By identification  $\iota: G\to \zhclass{w^{3^{q-r}}}$, we have that $\iota \psi = \psi_2^{q-r}$, where $\psi_2^r$ is listed in Lemma~\ref{selfdegreeT'}. Since $\dg \psi_2^{q-r}=(9-3^{2q+2})^{q-r}$, we obtain that $\dg \psi  = (9-3^{2q+2})^{q-r}/(8\cdot 3^{q-r}) = \frac{1}{8}\times (3-3^{2q+1})^{q-r} \tmod{3^{r}}$.
\end{proof}

\begin{Lemma}\label{selfdegreeD'}
Each endomorphism of $D'_{n\cdot 2^q}$ is conjugate to a composition of some endomorphisms in the following table
\begin{center}
\begin{tabular}{|c|lr|l|l|l|l|l|l|}
\hline
$\psi$ & $\ker \psi$ & & $\psi(u)$ & $\psi(w)$ &   $\im \psi$ &  $\dg(\psi)$
 \\ \hline
$\psi_{1,s}$ & $1$ &  $(s,n)=1, 1\le s <\frac{n}{2}$&  $u^{s}$ & $w$   &  $D'_{n\cdot 2^q}$ &
 $(1-n^{2^{q-1}}) s^2+ n^{2^{q-1}}$  \\ \hline
$\psi_{2,t}$ & $1$ &  $(t,2)=1$&  $u$ & $w^t$   &  $D'_{n\cdot 2^q}$ &
 $(1-n^{2^{q-1}})+ n^{2^{q-1}} t^2$   \\ \hline
$\psi_{3,1}$ & $\zhclass{u} \cong \bbZ_{n}$ &   &  $1$ & $w$   &  $\bbZ_{2^q}$ &
 $(1-n^{2^{q-1}})n^2+ n^{2^{q-1}}$   \\ \hline
$\psi_{3,k}$ & $\zhclass{u^k}\cong \bbZ_{\frac{n}{k}}$ &  $k|n, 1<k<n$&  $u^{\frac{n}{k}}$ & $w$   &  $D'_{k\cdot 2^q}$ &
 $(1-n^{2^{q-1}})\frac{n^2}{k^2}+ n^{2^{q-1}}$   \\ \hline
$\psi_{4}$ & $\zhclass{u, w^{2^{q-1}}}\cong \bbZ_{2n}$ &   &  $1$ & $w^{2}$
  &  $\bbZ_{2^{q-1}}$ &  $4n^{2^{q-1}}$  \\ \hline
\end{tabular}
\end{center}
\end{Lemma}

\begin{proof}
By Lemma~\ref{subgroupD2q}, $\ker \psi$ must be $1$,  $\zhclass{u^k}$ ($k|n, 1\le k<n$), $\zhclass{uw^{2^r}}$ ($0< r <q$), $\zhclass{u^kw^{2^r}}$ ($k|n, 1< k<n$, $1< r <q$), or $D'_{n\cdot 2^q}$.

(1) $\ker \psi =1$. Since $\psi(u)$ has order $n$ and $\zhclass{u}$ is the commutator and hence is a fully invariant subgroup, $\psi(u)=u^s$ for some $s$ with $(s,n)=1$. Since $\psi(w)$ has order $2^q$, we have $\psi(w)=u^mw^t$ for some odd $t$. Note that $u(u^m w^t)u^{-1}=u^{m+2}w^t$. Up to a conjugation, we may assume that $\psi(w)=w^t$. Thus, $\psi= \psi_{1,s}\psi_{2,t}$. Moreover, since $w(u^s w^t)w^{-1}=u^{-s}$, we may assume that $1\le s <\frac{n}{2}$.

(2) $\ker \psi =\zhclass{u^k}$ with $k|n$ and $1\le k<n$. Then $\psi(u)=u^{\frac{n}{k}}$. Since $\psi(w)$ has order $2^{q}$, by Lemma~\ref{conjD2q}, we have $\psi(w)=w^{t}$ for some odd $t$. Thus, $\psi$ is $\psi_{2,t}\psi_{3,k}$.

(3) $\ker \psi =\zhclass{uw^{2^r}}$ with $0<r<q$. Then $\psi(u)=1$. Since $\psi(w)$ has order $2^r$, by Lemma~\ref{conjD2q}, we have $\psi(w)=w^{t\cdot 2^r}$ for some odd $t$. Thus, $\psi$ is $\psi_{2,t}\psi_{4}^r$ if $0<r<q$.

(4) $\ker\psi = D'_{n\cdot 2^q}$. Clearly, $\psi=\psi_4^q$.

Now we consider the degree $\dg(\psi)$. Note that both $\zhclass{u}$ and $\zhclass{w}$ are invariant subgroups of $\psi$. Assume that $\psi(u)=u^s$ and $\psi(w)=w^t$. By Lemma~\ref{deglenspace}, we have that $\dg(\psi|_{\zhclass{u}})=s^2$ and $\dg(\psi|_{\zhclass{w}})=t^2$. Thus, $\dg(\psi)$ is the unique solution module $n\cdot 2^q$ of the system of congruences:
$$ x \equiv s^2 \tmod{n}, \ \ x \equiv t^2 \tmod{2^q}.$$
By Euler-Fermat theorem, we have that $n^{\varphi(2^q)}=n^{2^{q-1}}\equiv 1 \tmod{2^q}$. Thus, we have
$n\cdot n^{2^{q-1}-1}+2^{-q}(1-n^{2^{q-1}})2^q=1$. By the Chinese remainder theorem, we have
$\dg(\psi) = (1-n^{2^{q-1}}) s^2+n^{2^{q-1}} t^2$.
We obtain the all $\dg(\psi)$'s.
\end{proof}

\begin{corollary} The group $Out(D'_{n\cdot 2^q})$ is isomorphic to
$(\mathbb Z_n^*/\mathbb Z_2)\times \mathbb Z_{2^q}^*$.
\end{corollary}
\begin{proof} Here we give an alternate proof of this reuslt as \cite[Theorem 2.4(2)]{GoGo} contains an error which will be corrected. From Proposition 1.3 of \cite{GoGo} we have a short exact sequence,
$$1\to Der_{\alpha}(\mathbb Z_{2^q} , \mathbb Z_{n}) \to
 Aut(\mathbb Z_{n} \rtimes _{\alpha} \mathbb Z_{2^q})\to \eta(\mathbb Z_{n}\rtimes_{\alpha} \mathbb Z_{2^q})\to 1$$
where  $Der_{\alpha}( \mathbb Z_{2^q}, \mathbb Z_{n})\cong \mathbb Z_{n}$, and
$\eta(\mathbb Z_{n}\rtimes_{\alpha} \mathbb Z_{2^q})\cong \mathbb Z_n^*\times \mathbb Z_{2^q}^*$. Since the group
$\mathbb Z_n^*\times \mathbb Z_{2^q}^*$
embeds naturally in  $Aut(\mathbb Z_{n} \rtimes _{\alpha} \mathbb Z_{2^q})$ we have a section. Thus, we obtain
$$ Aut( \mathbb Z_{n} \rtimes _{\alpha} \mathbb Z_{2^q})\cong \mathbb Z_{n}\rtimes (\mathbb Z_n^*\times \mathbb Z_{2^q}^*).$$
Denote by $(a,(b,c))$ a typical element of $Aut( \mathbb Z_{n} \rtimes _{\alpha} \mathbb Z_{2^q})$ where $a\in \mathbb Z_{n}, (b,c)\in \mathbb Z_n^*\times \mathbb Z_{2^q}^*$. Here, we write $\mathbb Z_k$ as an additive group and $\mathbb Z_k^*$ as a multiplicative group.
Now let us factor this group by the inner automorphisms. Write $\langle x \rangle =\mathbb Z_n$ and $\langle y\rangle = \mathbb Z_{2^q}$.

The conjugation of $(x,0)$ by $(0,y)$ yields $(-x,0)$ thus the corresponding inner automorphism is the element $(0, (-1,1))$. The conjugation of $(0,y)$ by $(x,0)$ yields $(2x,y)$ thus the corresponding inner automorphism is the element $(2, (1,1))$. It follows that the inner automorphisms form the subgroup
$\mathbb Z_{n}\rtimes (\mathbb Z_2\times \{1\}).$ Therefore the quotient is isomorphic to
$(\mathbb Z_n^* /\mathbb Z_2)\times \mathbb Z_{2^q}^*$
and the result follows.
\end{proof}

\begin{theorem}\label{domain=dicyclic}
Let $\varphi$ be a surjective non-trivial homomorphism from $D'_{n\cdot 2^q}$ to $G$ induced by a map from $S^3/D'_{n\cdot 2^q}$ to some $3$-dimensional spherical manifold $S^3/G$. Then $\varphi = \eta'\psi \eta''$, where $\eta''\in Aut(D'_{n\cdot 2^q})$ and $\eta'\in Aut(G)$. Moreover, all possible $G$ and $\psi$ are listed as follows:
\begin{center}
\begin{tabular}{|l|l|l|l|l|l|l|}
\hline
$G$ & $S^3/G$ & $\psi(u)$ & $\psi(w)$ &    $\dg(\psi)$ \\ \hline
$\bbZ_{2^q}$ & $L(2^{q}; 1+2^{q-1}, 1)$
 & $1$ & $c$ & $(1-2^{q-1})n+n^{2^{q-1}-1}$ \\ \hline
$\bbZ_{2^{r}}$ \ $0<r<q$ & $L(2^{r}; 1, 1)$
 & $1$ & $c$ & $2^{q-r}n^{(q-r)2^{q-1}-1}$ \\ \hline
$D'_{k\cdot 2^r}$  \ $k|n, k>1, 1<r<q$ & $S^3/D'_{k\cdot 2^r}$
 & $\bar u$ & $\bar w$ & $\frac{n}{k}(1-k^{2^{r-1}})2^{r-q}(1-k^{2^{q-r-1}})$\\
 & & & & $\ \ +k n^{2^{r-1}-1} 2^{q-r}$  \\ \hline
$D'_{k\cdot 2^q}$ & $S^3/D'_{k\cdot 2^q}$
 & $\bar u$ & $\bar w$ &  $(1-n^{2^{q-1}})\frac{n}{k}+ kn^{2^{q-1}-1}$ \\ \hline
\end{tabular}
\end{center}
\end{theorem}

\begin{proof}
If  $\ker \psi = \zhclass{u^kw^{2^r}}\cong\bbZ_{\frac{n}{k}\cdot 2^{q-r}}$ ($k|n$, $k>1$ and $1<r<q$). The homomorphism $\psi: D'_{n\cdot 2^q}\to D'_{k\cdot 2^r}$ can be given by $\psi(u) = \bar u$, $\psi(w) = \bar w$. By Lemma~\ref{alllensspace},  $S^3/\zhclass{w}$ and $S^3/\zhclass{\bar w}$ are respectively lens spaces $L(2^q; 1+2^{q-1},1)$ and  $L(2^r; 1+2^{r-1},1)$. By Lemma~\ref{deglenspace}, $\dg\psi|_{\zhclass{w}} \equiv 2^{q-r}(1+2^{r-1})[(1+2^{q-1})]_{(2^q)}^{-1} = 2^{q-r}(1+2^{r-1})(1-2^{q-1})\equiv 2^{q-r} \tmod{2^r}$. On the other hand, by Lemma~\ref{alllensspace} again,  $S^3/\zhclass{u}$ and $S^3/\zhclass{\bar u}$ are respectively lens spaces $L(n; 1,-1)$ and  $L(k; 1,-1)$. By Lemma~\ref{deglenspace}, $\dg\psi|_{\zhclass{u}} \equiv 1 \tmod{k}$. It follows that
$$
\begin{array}{rcl}
\dg(\psi)
 & = & (1-k^{2^{r-1}})\frac{n}{k}[2^{q-r}]_{(k)}^{-1} +  k^{2^{r-1}} 2^{q-r} [\frac{n}{k}]_{(2^r)}^{-1}\\
 & = & (1-k^{2^{r-1}})\frac{n}{k}2^{r-q}(1-k^{2^{q-r-1}})+k^{2^{r-1}} 2^{q-r} (\frac{n}{k})^{2^{r-1}-1}\\
 & \equiv & (1-k^{2^{r-1}})\frac{n}{k}2^{r-q}(1-k^{2^{q-r-1}})+k n^{2^{r-1}-1} 2^{q-r}  \tmod{k\cdot 2^r} \\
\end{array}
$$

The other cases come from Lemma~\ref{selfdegreeD'}, by using Lemma~\ref{degfromself}.
\end{proof}

Next Lemma deals with the case where the domain is in the form of a direct product.

\begin{Lemma}
Let $\varphi: \bbZ_m \times G_1\to G'$ be a surjective homomorphism. Then $G'=\bbZ_{k}\times G_2$ and $\varphi= \varphi_1\times \varphi_2$. Moreover, both of $\varphi_1: \bbZ_m\to \bbZ_k$ and $\varphi_2: G_1\to G_2$ are surjective, and $\dg \varphi = (|G_1|\dg \varphi_2 + m \dg \varphi_1) [|G_1|+m]^{-1}_{(m|G_1|)}$.
\end{Lemma}

\section{Computations}

We conclude this paper with some computations of the mapping degree of certain spherical $3$-manifolds.

In order to compute $D(S^3/G_1, S^3/G_2)$, our strategy is as follows: (1) look for all normal subgroups $H$ of $G_1$ such that $G_1/H$ is isomorphic to a subgroup of $G_2$; (2) find the degree of $\varphi_H: G_1\to G_1/H$; (3) look for all embedding $\xi: G_1/H\to G_2$. The degrees of epimorphisms $\varphi_H:G_1\to G_1/H$ were computed in
Lemma 3.3 if $G_1$ is cyclic, and in Theorems 4.4, 4.5, 4.6, 4.9, 4.12 if $G_1$
is one of the groups listed in Table 2.1.

Thus Lemma~\ref{degfromself} will compute the set 
%$\bar degree psi$.$\bar  $D(S^3/G_1, S^3/G_2)$ will be
 $\{\dg(\xi)\dg(\varphi_H)\}$. Here, $\dg(\xi) = \dg(\eta')|G_2||H|/|G_1|$ for some $\eta'\in Out(G_1/H)$. The set $D(M,N)$ can be determined once
$\overline{\deg}(\psi)$ is for all $\psi$. 
%Then Lemma 4.1 can be used to determine this degree $\overline{deg}$.}

%I think that we can say that the set $D(M,N)$ can be determined once
%$\overline{deg}(\psi)$ is. Then Lemma 4.1 can be used to determine this
%degree $\overline{deg}$.

\begin{example}
Let $G_1=D'_{n\cdot 2^q}$ and $G_2=D'_{m\cdot 2^r}$ with $r<q$.
\end{example}

Suppose $f:M_1\to M_2$ is a map between two spherical $3$-manifolds with $\pi_1(M_1)=G_i$ for $i=1,2$ and $\varphi=f_{\sharp}: G_1\to G_2$.
By Lemma~\ref{subgroupD2q}, if $\ker \varphi$ is non-trivial, then $\ker \varphi$ is either  $\bbZ_{\frac{n}{k}\cdot 2^{q-r}}$ with $k|(n,m)$ or $\bbZ_{n\cdot 2^{l}}$ with $1< l \le r$.

(1) $\ker \varphi = \zhclass{u^kw^{2^r}}\cong \bbZ_{\frac{n}{k}\cdot 2^{q-r}}$ with $k|(n,m)$. By Theorem~\ref{domain=dicyclic}, we have $$\dg(\varphi) = \frac{m}{k}\cdot \dg(\eta')(\frac{n}{k}(1-k^{2^{r-1}})2^{r-q}(1-k^{2^{q-r-1}})+k n^{2^{r-1}-1} 2^{q-r}) \dg(\eta''),$$
where $\eta''\in Aut(D'_{n\cdot 2^q})$ and $\eta'\in Aut(D'_{q\cdot 2^r})$.

(2) $\ker \varphi = \zhclass{uw^{2^{q-l}}}\cong \bbZ_{n\cdot 2^{l}}$ with $1< l \le r$. By Theorem~\ref{domain=dicyclic}, we have $$\dg(\varphi) = 2^{r-l}\dg(\eta')(2^{q-l}n^{(q-l)2^{q-1}-1}) \dg(\eta''),$$
where $\eta''\in Aut(D'_{n\cdot 2^q})$ and $\eta'\in Aut(D'_{q\cdot 2^r})$.

Clearly, if $\ker \varphi= D'_{n\cdot 2^q}$, then $\dg(\varphi) \equiv 0 \tmod{m\cdot 2^r}$.  Thus, we can write down the set $D(S^3/D'_{n\cdot 2^q}, S^3/D'_{m\cdot 2^r})$.

\begin{example}
Let $G_1=\bbZ_{120}$ and $G_2=O^*_{48}\times \bbZ_5$.
\end{example}

By Lemma~\ref{alllensspace}, the manifold $S^3/(O^*_{48}\times \bbZ_5)$ is covered by three lens spaces $L(30; 1, 11)$, $L(40; 3,13)$ and $L(40; -7, 23)$. Thus, any map from $L(120; r_1, r_2)$ can be lifted to a map into one of these three lens spaces.

Note that the degrees from these three lens spaces to $S^3/(O^*_{48}\times \bbZ_5)$ are respectively $8$, $6$ and $6$. By Theorem \ref{degreebetweenlensspaces}, $D(L(120; r_1, r_2), S^3/(Q_{48}\times \bbZ_5))$ consists of the following three sets:
$$
\begin{array}{l}
\{j^2\cdot 1\cdot 11 \frac{120\cdot 30}{30^2}[r_{1}]^{-1}_{(120)}[r_{2}]^{-1}_{(120)}\cdot 8 + 240k \mid k\in \mathbb Z, j=0,1,\ldots, 29\} \\
\{j^2\cdot 3\cdot 13 \frac{120\cdot 40}{40^2}[r_{1}]^{-1}_{(120)}[r_{2}]^{-1}_{(120)}\cdot 6 + 240k \mid k\in \mathbb Z, j=0,1,\ldots, 39\} \\
\{j^2\cdot (-7)\cdot 23 \frac{120\cdot 40}{40^2}[r_{1}]^{-1}_{(120)}[r_{2}]^{-1}_{(120)}\cdot 6 + 240k \mid k\in \mathbb Z, j=0,1,\ldots, 39\} \\
\end{array}
$$
Note that any lens space $L(120; r_1, r_2)$ must be homeomorphic to one of the form  $L(120; 1, r)$, where
$$r\in\{1, 7, 11, 13, 19, 23, 29, 31, 41, 43, 49, 59\}.$$
We list all the elements of $D(L(120; r_1, r_2),\ S^3/(O^*_{48}\times \bbZ_5))$ as follows.
\begin{center}
\begin{tabular}{|l|l|l|}
  \hline
  $r$ &  $D(L(120; 1, r), S^3/(O^*_{48}\times \bbZ_5))$\\
  \hline
 $1$ & $\{0, 30, 48, 72, 78, 112, 120, 160, 168, 192, 208, 222\}$ \\ \hline
 $7$ & $\{0, 16, 24, 64, 66, 96, 114, 120, 144, 160, 210, 216\}$ \\ \hline
$11$ &  $\{0, 32, 42, 48, 72, 80, 90, 120, 128, 138, 168,
   192\}$\\ \hline
$13$ & $\{0, 6, 16, 24, 54, 64, 96, 120, 144, 150, 160,
   216\}$ \\ \hline
$19$ &  $\{0, 42, 48, 72, 90, 112, 120, 138, 160, 168, 192,
   208\}$ \\ \hline

$23$ &  $\{0, 24, 66, 80, 96, 114, 120, 144, 176, 210, 216,
   224\}$  \\ \hline

$29$ &  $\{0, 32, 48, 72, 80, 102, 120, 128, 150, 168, 192,
   198\}$  \\ \hline

$31$ & $\{0, 18, 48, 72, 112, 120, 160, 162, 168, 192, 208,
    210\}$ \\ \hline

$41$ & $\{0, 30, 32, 48, 72, 78, 80, 120, 128, 168, 192,
   222\}$ \\ \hline

$43$ &  $\{0, 16, 24, 64, 90, 96, 120, 144, 160, 186, 216,
   234\}$  \\ \hline

$49$ & $\{0, 30, 48, 72, 78, 112, 120, 160, 168, 192, 208,
   222\}$ \\ \hline

$59$  &  $\{0, 32, 42, 48, 72, 80, 90, 120, 128, 138, 168,
   192\}$ \\ \hline
\end{tabular}
\end{center}

\end{document}